\RequirePackage{etex}

\documentclass[12pt,psamsfonts]{amsart}
\usepackage{amsfonts,pstricks}
\usepackage{amsmath, amsthm, amssymb,  mathrsfs, epsfig}
\usepackage{dcpic, pictex, pinlabel}
\usepackage{hyperref,url}
\usepackage{graphicx}
\usepackage{color}
\usepackage{ifpdf}
\usepackage{tikz}
\usepackage{xypic}
\usepackage{extarrows}
\usepackage{multirow}
\usepackage{bm}
\usepackage[foot]{amsaddr}

\DeclareMathOperator{\rk}{rank}
\DeclareMathOperator{\spn}{span}

\begin{document}
\newcommand{\Z}{\mathbb{Z}}
\newcommand{\R}{\mathbb{R}}
\newcommand{\N}{\mathbb{N}}
\newcommand{\M}{\mathcal{M}}
\newcommand{\B}{\mathcal{B}}
\newcommand{\Hilb}{\mathcal{H}}
\newcommand{\C}{\mathbb{C}}
\newcommand{\Sen}{\mathcal{S}}
\newcommand{\A}{\mathcal{A}}
\newcommand{\Q}{\mathbb{Q}}
\newcommand{\I}{\mathcal{I}}
\newcommand{\T}{\mathcal{T}}
\newcommand{\MF}{\operatorname{MF}}
\newcommand{\QMF}{\operatorname{QMF}}
\newcommand{\MC}{\operatorname{MC}}
\newcommand{\QMC}{\operatorname{QMC}}
\newcommand{\GQSAT}{\operatorname{Gen-QSAT}}
\newcommand{\EE}{\operatorname{EE}}
\newcommand{\MEE}{\operatorname{MEE}}

\theoremstyle{definition}

\newtheorem{thm}{Theorem}[section]
\newtheorem{axiom}[thm]{Axiom}
\newtheorem{conjecture}[thm]{Conjecture}
\newtheorem{lem}[thm]{Lemma}
\newtheorem{example}[thm]{Example}
\newtheorem{cor}[thm]{Corollary}
\newtheorem{prop}[thm]{Proposition}
\newtheorem{rem}[thm]{Remark}
\newtheorem{definition}[thm]{Definition}
\newtheorem{measurement}[thm]{Measurement}
\newtheorem{ancilla}[thm]{Ancilla}
\newtheorem{note}[thm]{Note}
\newtheorem{claim}[thm]{Claim}

\numberwithin{equation}{section} \makeatletter
\renewenvironment{proof}[1][\proofname]{\par
    \pushQED{\qed}%
    \normalfont \topsep6\p@\@plus6\p@ \labelsep1em\relax
    \trivlist
    \item[\hskip\labelsep\indent
        \bfseries #1]\ignorespaces
}{%
    \popQED\endtrivlist\@endpefalse
} \makeatother
\renewcommand{\proofname}{Proof}

\newcommand{\qsat}{\textsc{qsat}}
\newcommand{\bra}[1]{\langle #1|}
\newcommand{\ket}[1]{|#1\rangle}
\newcommand{\ketbra}[2]{\ket{#1}{\bra{#2}}}
\newcommand{\onote}[1]{\textcolor{blue}{ {\textbf{(Or:}
#1\textbf{) }}}}
% add comment to next line if we want onotes
%\renewcommand{\onote}[1]{}

\title{Quantum Max-flow/Min-cut}
\author{Shawn X. Cui$^{1,5}$ \and Michael H. Freedman$^{1,2}$ \and Or Sattath$^{4}$ \and Richard Stong$^{3}$ \and Greg Minton$^3$}

\address{$^1$Department of Mathematics\\University of California\\Santa Barbara, CA 93106}
\email{xingshan@math.ucsb.edu}
\address{$^2$Microsoft Research, Station Q\\ University of California\\ Santa Barbara, CA 93106}
\email{michaelf@microsoft.com}
\address{$^3$Center for Communications Research, La Jolla, CA 92121, USA}
\email{stong@ccrwest.org}
\email{gtminto@ccrwest.org}
\address{$^4$Computer Science Division, University of California,
Berkeley, CA 94720}
\email{sattath@gmail.com}
\address{$^5$Quantum Architectures and Computations Group, Microsoft Research, Redmond, WA}
\date{\today}

\begin{abstract}
The classical max-flow min-cut theorem describes transport through certain idealized classical networks. We consider the quantum analog for tensor networks. By associating an integral capacity to each edge and a tensor to each vertex in a flow network, we can also interpret it as a tensor network, and more specifically, as a linear map from the input space to the output space. The quantum max flow is defined to be the maximal rank of this linear map over all choices of tensors. The quantum min cut is defined to be the minimum product of the capacities of edges over all cuts of the tensor network. We show that unlike the classical case, the quantum max-flow=min-cut conjecture is not true in general. Under certain conditions, e.g., when the capacity on each edge is some power of a fixed integer, the quantum max-flow is proved to equal the quantum min-cut. However, concrete examples are also provided where the equality does not hold.

%Although some aspects generalize, surprising counterexamples are found.
We also found connections of quantum max-flow/min-cut with entropy of entanglement and the quantum satisfiability problem. We speculate that the phenomena revealed may be of interest both in spin systems in condensed matter and in quantum gravity.
\end{abstract}

\maketitle

\section{Introduction}
% \onote{Why $QMF(\cdot)$ and not simply $QMF(\cdot)$? Same remark for $QMC(\cdot)$.}
% \onote{How about adding a few words about background and uses of tensor networks? Probably we should mention MPS, PEPS and MERA. And that the probability of acceptance of a quantum circuit is the contraction of a tensor network. Is there a good survey about this topic?}
Networks transport classical things like power, water, oil, and cars.  Tensor networks transport linear algebraic things like rank and entanglement and should be thought of as the quantum analogy.  We take a first step in comparing the two.  In 1956 two papers \cite{elias1956note, ford1956maximal} set the classical study on a strong algorithmic foundation by proving the max-flow=min-cut theorem $(\MF/\MC)$ which, roughly speaking, says that in a certain idealized limit, capacity or ability of a network to transport is equal to a measure of what needs to be cut to totally sever the network.  This paper explores the quantum analogy of $\MF/\MC$ for tensor networks.

%A symmetric flow network with integral capacities can also be interpreted as a tensor network, provided one associates a tensor (with the correct dimensions) to each node.

%Tensor networks often appear in physics in the following context. It is impossible to optimize a function over all possible states (operators)  because the Hilbert space is too big. Instead, one chooses a class of states (operators) which can be described by a tensor network which has only $poly(n)$ parameters, where $n$ is the size of the system (for example - number of particles or number of sites in a lattice). This approximation turns out to be useful both analytically and numerically: An important special case is the class of Matrix Product States (MPS) which are known to faithfully represent the ground state of a gapped 1-D Local-Hamiltonian;  DMRG is a successful algorithm which optimizes over MPS. Ref.~\cite[Section 5]{gharibian2014quantum} provides a gentle introduction to this topic.

Tensor networks have been extensively studied in physics literature, especially in condensed matter physics and quantum gravity. In quantum many-body systems, the ground states can be represented as tensor networks whose complexity is typically a polynomial of parameters (for example - in terms of the number of particles or number of sites in a lattice), instead of the exponential of parameters under the naive representation. Moreover, it is more convenient to visualize the entanglement entropy of a many-body system using tensor networks. The area law naturally provides an upper bound for the entanglement entropy. Among important classes of tensor networks are Matrix Product States (MPS) \cite{fanns1992} in $1d$ which are known to faithfully represent the ground state of a gapped 1-D Local-Hamiltonian and Projected Entangled Pair States (PEPS) \cite{verstrate2004} in $2d$. For an introduction on tensor networks, see \cite{gharibian2014quantum}, \cite{orus2014}, etc. Connections of tensor networks to holographic duality have also been proposed \cite{maldacena1998} \cite{ryu2006}. Perfect tensors are used to construct holographic states and holographic codes \cite{pastawski2015}. Under certain conditions, the area law is shown to be saturated and the Ryu-Takayanagi formula holds.

One motivation for the study of quantum max-flow min-cut comes from \cite{calegari2010}. In \cite{calegari2010}, the authors considered tensor networks where all edges have the same dimension $k$ and all vertices are assigned the same tensor. It was conjectured that the maximal rank of the tensor network (quantum max-flow) is equal to $k$ raise to the power of classical min-cut (quantum min-cut). If the conjecture were true, it would imply the existence of tensors with certain \lq\lq positive'' properties, the construction of which is part of the work in showing the positivity of the universal pairings in unitary $(2+1)$-TQFTs.

In this paper, we generalize the conjecture in two versions. The first version is more general than the second version, and the original conjecture is a special case of the second version. We show that the quantum $\MF/\MC$ conjecture in both of the two versions does not hold in general, \footnote{Actually the negative of the first version implies the negative of the second version.} but we will give some conditions under which the first version does hold. More detailed results will be given after introducing some terminology below.

Associated with each tensor network is an undirected graph which has some internal edges and open edges. All edges are assigned an integral capacity (dimension) and all vertices are assigned a tensor. In the first version, the tensors at different vertices are independent of each other. All of them are chosen arbitrarily. In the second version, vertices of the same {\it valence type} are required to be assigned the same tensor, where two vertices have the same valence type if they have the same degree and the same sequence of capacities of edges adjacent to each of them (See Section \ref{sec:second}). In particular, if all edges have the same capacity and all vertices have the same degree, then all the vertices also have the same valence type and thus are assigned an identical tensor, which reduces to the original requirement in \cite{calegari2010}. Apparently, the choice of tensors in the second version are more restricted than in the first version.

In either version, we partition the set of open edges into two disjoint subsets called the input set and the output set, and define the input space (resp. output space) to be the tensor product of the Hilbert spaces associated to each edge in the input set (resp. output set). Contracting the tensor network along internal edges results in a linear map $L$ from the input space to the output space. We define the quantum flow as the rank of this linear map, and the quantum max-flow is the maximum value that the quantum flow can take. An edge cut set is a set of edges, the removal of which disconnects the input from the output. The cut value is defined to be the product of the capacities of all edges in an edge cut set, and the quantum min-cut is the minimum value among all the cut values. Every cut provides an interpretation of the linear map $L$ as $L = L_2 L_1$ where the dimension of the intermediate space is the cut value, and therefore the quantum min-cut provides an upper-bound on the quantum max-flow. By definition, the quantum max-flow in the first version is no less than that in the second version. In the following we focus on the quantum $\MF/\MC$ of the first version.

%Each cut defines a partition so that $A=BC$ where $B$ is the linear map defined by the same operator until the cut, and $C$ is the linear map defined from

%Each cut defines a concatenation of two linear maps which is equal to the original linear map. The intermediate space gives an upperbound on the rank. We define the quantum min cut as the smallest such value.

%The rank of this linear map has a typical value (which is the maximal value among all possible choices - see Proposition~\ref{prop:open dense), which we define as the maximum quantum flow.

We find that a bit of elementary number theory enters.  When the local degrees of freedom (i.e. capacities) are organized in finite $d_i = d^{k_i}$ dimensional Hilbert spaces for fixed $d$, then there is a straightforward generalization (Theorem \ref{thm:max=min}) of the classical $\MF/\MC$.  However, if - to take the other extreme - the various dimensions $\{d_i\}$ are relatively prime, then new and surprising phenomena are seen.

Already in the case where some Hilbert spaces, bonds of a tensor network, have $\dim = 2$ and others have $\dim = 3$, one observes a surprising drop in \lq\lq capacity'' which in this context means either rank (Section~\ref{sec:first}), entropy of entanglement (Section ~\ref{sec:entanglement}), or the dimension of the unsatisfying subspace of a quantum satisfiability instance (Section~\ref{sec:qsat}).  The lowest possible dimension for this phenomenon is Example \ref{eg:rk7 8}(also see Figure \ref{fig:rk7copy}) in which a tensor network which on the basis of \lq\lq cut reasoning'' appears to have maximal rank $=8$ actually has maximal rank $=7$. Thus for this network, the quantum max flow is strictly less than the quantum min cut. Most of our results can be summarized in the following meta-statement:

\begin{thm}[Main Result, informal] The quantum max flow is at most the quantum min cut. There exist examples where this inequality is strict, and other examples where this inequality becomes an equality.
\end{thm}

We really only scratch the surface in this note and cannot yet answer the obvious questions about typical or asymptotic behaviors of large random networks with bonds of relatively prime dimensions, although we do give some elementary lower bounds (Proposition \ref{prop:lower_bound}).

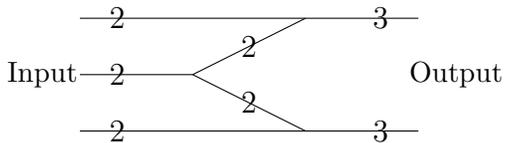
\begin{figure}
\begin{tikzpicture}[scale = 0.5]
\draw (0,0) -- (9,0) ;
\draw (0,3) -- (9,3) ;
\draw (0,1.5) -- (3,1.5) -- (6,3) ;
\draw (3,1.5) -- (6,0) ;

\draw (1,0) node{$2$};
\draw (1,1.5) node{$2$};
\draw (1,3) node{$2$};
\draw (8,0) node{$3$};
\draw (8,3) node{$3$};
\draw (4.5,0.75) node{$2$};
\draw (4.5,2.25) node{$2$};

\draw (-1,1.5) node{{\small{Input}}};
\draw (10,1.5) node{{\small{Output}}};
\end{tikzpicture}
\caption{maximal rank = 7}\label{fig:rk7copy}
\end{figure}

Our examples suggest two lines for future investigations: Example (\ref{eg:rk7 8}) can be read (in light of Section \ref{sec:entanglement}) as revealing an unexpected reluctance of spin $\frac{1}{2}$ and spin $1$ particles to entangle.  At least when coupled by that network, regardless of the tensor coupling, entanglement cannot be maximal.  In an entirely different direction, the capacity of tensor networks may have something to say about quantum gravity.  Entanglement entropy in the holographic side of AdS/CFT duality has been recognized as equivalent to the minimal area (on a related geometric functional) of cut surfaces on the AdS side (\cite{headrick2014hep-th} \cite{ryu2006}).  It is natural to interpret the cut surface dually as a transverse flow (\cite{work-in-progress}) with the flow lines being strands of entanglement.  It would be natural to go further and replace this overly classical ¡°entanglement flow¡± with something more quantum: a tensor network.  If one postulated that fundamental degrees of freedom are finite and not all commonly divisible, then the present paper reveals that entanglement may be unexpectedly small.  Since maximal entanglement, for example between infalling states and Hawking radiation emitted by black holes beyond their ¡°Page time,¡± is central to the ¡°firewall¡± paradox (\cite{almheiri2012hep-th}), a mechanism which reduces entanglement is of potential interest.

\paragraph{\bf Structure.} In Section \ref{sec:classical}, we give a review of the classical max-flow min-cut theorem and Menger's theorem. In Section \ref{sec:first} and \ref{sec:second}, we provide two versions of quantum analogues of the max-flow min-cut theorem.
%  \onote{Strictly speaking - AFAIK this is not a generalization: There is no special case which gives you the classical result. Perhaps use the term analogue / analogous instead?}
Unlike the classical case, we prove that the quantum max-flow min-cut theorem only holds for some networks. A number of examples are explained. Section \ref{sec:entanglement} studies the quantum max-flow min-cut theorem from the perspective of entanglement entropy. In Section~\ref{sec:qsat} we show a relationship between quantum max flow and the quantum satisfiability problem.

\section{Classical max-flow min-cut theorem}
\label{sec:classical}

The classical max-flow min-cut theorem was proven by Elias, Feinstein and Shannon \cite{elias1956note} in $1965$, and independently also by Ford and Fulkerson \cite{ford1956maximal}. The theorem states that the maximum amount of the flow in a network from the source to the sink is equal to the minimum capacity that, when removed from the network, causes no flow to pass from the source to the sink. We first give several definitions below and then state the theorem in detail.

Let $G = (V,E)$ be a directed graph $($flow network$)$ where $V$ is the set of vertices and $E$ is the set of edges. Let $S$ and $T$ be the set of sources and sinks, respectively. Namely, $S \ (resp. \; T)$ is the set of vertices with only outgoing $(resp. \; $incoming$)$ edges. The capacity of the edges is a function $c: E \longrightarrow \R^{+}$ such $c_e$ gives the maximum amount of flow through each edge $e \in E$.

A flow from the sources to the sinks is given by a function $f: E \longrightarrow \R^{+}$, such that $f$ satisfies:
% \onote{This is a weird notation. Why not 1) instead of "1)."  ? Also, why don't you use begin\{enumerate\} }
\begin{enumerate}
\item  capacity constraint:
$$f_{ij} \leq c_{ij},\, \forall (i,j) \in E,$$
\item conservation of flow:
$$\sum\limits_{\{i: (i,j) \in E\}} f_{ij} = \sum\limits_{\{k: (j,k) \in E\}} f_{jk}, \, \forall j \in V \setminus (S \cup T).$$
\end{enumerate}

% $1).$ capacity constraint:
% $$f_{ij} \leq c_{ij},\, \forall (i,j) \in E,$$

% $2).$ conservation of flow:
% $$\sum\limits_{\{i: (i,j) \in E\}} f_{ij} = \sum\limits_{\{k: (j,k) \in E\}} f_{jk}, \, \forall j \in V \setminus (S \cup T).$$

\begin{definition}
The value of the flow $f$ in a network $G = (V,E)$ with the capacity function $c$ is defined to be $|f| = \sum\limits_{\{(i,j) \in E: i \in S\}} f_{ij}$. The maximum amount of flow $\MF(G,c)$ is defined to be $\max \{|f|: f \textrm{ is a flow}\}$.
\end{definition}

An edge cut set $C$ is a set of edges such that there exists a partition $V = \bar{S} \sqcup \bar{T}$ such that $S \subset \bar{S}, T \subset \bar{T}$, and $C = \{(u,v) \in E: u \in \bar{S}, v \in \bar{T} \}$. Clearly, the removal of the edges in $C$ from $E$ disconnects all paths from $S$ to $T$. Note that there could still be paths from $T$ to $S$ after the edges in $C$ are removed.

%An $S$-$T$ cut of $N$ is a partition of $V$ into $\bar{S} \sqcup \bar{T}$, such that $S \subset \bar{S}, T \subset \bar{T}$.
The following theorem is well-known in graph theory.
\begin{definition}
The capacity of an edge cut set $C$ in a network $G = (V,E)$ with the capacity function $c$ is defined to be $|C| = \sum\limits_{\{(i,j) \in C\}} c_{ij}$. We define the min-cut $\MC(G,c) = \min \{|C|: C \textrm{ is an edge cut set}\}$.
\end{definition}

\begin{thm} \cite{elias1956note}\cite{ford1956maximal} \label{thm:classical max min}
$[$Max-flow Min-cut Theorem$]\\$ For a network $G = (V,E)$ with the capacity function $c: E \longrightarrow \R^{+}$, the maximum amount of flow $\MF(G,c)$ from the sources to the sinks is equal to the minimum capacity $\MC(G,c)$.
\end{thm}

If the capacity of every edge is a rational number, then the Ford-Fulkerson algorithm \cite{ford1956maximal} provides an efficient way to construct the max flow. Moreover, if all the capacities are integers, the max flow resulting from the Ford-Fulkerson algorithm also has integral values at every edge.

Thus in particular, when the capacity is $1$ on every edge, the maximum amount of flow is equal to the maximum number of edge disjoint directed paths from a source to a sink, and the max-flow min-cut theorem reduces to the directed Menger's theorem \cite{menger1927allgemeinen}. Furthermore, the undirected Menger's theorem can also be derived as a special case. Since we will mainly generalize this case to the quantum network, it is worth stating this theorem in more detail.

 Assume $G = (V,E)$ is an undirected graph with a specified partition $S \sqcup T$ of the set of degree $1$ vertices, where $S$ and $T$ are called sources $($or inputs$)$ and sinks $($or outputs$)$, respectively. Let $\MF(G)$ be the maximum number of edge disjoint paths in $G$ connecting a vertex in $S$ to a vertex in $T$, and let $\MC(G)$ be the minimum cardinality of all edge cut sets where an edge cut set is defined in the same way as in the case of directed graph. Note that here \lq\lq edge disjoint'' means that paths \emph{are} allowed to share vertices but not edges.

\begin{thm}\cite{menger1927allgemeinen}\cite{elias1956note}\cite{ford1956maximal}\label{thm:Menger}$[$Undirected Menger's Theorem$]\\$
Let $G =(V,E)$ be as above, then $\MF(G)$ $=$ $\MC(G)$.
\begin{proof}
Let $C$ be an edge cut set such that $|C| = $ $\MC(G)$. Since the removal of $C$ from $E$ disconnects $S$ from $T$, every path connecting some input to some output must contains at least one edge from $C$, and different edge disjoint paths contain different edges from $C$. Therefore $\MF(G)$ $\leq$ $\MC(G)$.

We turn $G$ into a directed graph $G'$ and use the max-flow min-cut theorem to prove $\MF(G) \geq \MC(G)$. Start with a new graph with the same set of vertices as $G$, but with no edges. For each edge $(i,j) \in E$, insert a pair of directed edges $(i,j), (j,i)$ to the new graph, then remove all the edges from the new graph which come into the set $S$ or leave the set $T$. Denote the resulting graph by $G'$. By construction $G'$ also has the inputs $S$ and outputs $T$. Define the capacity function $c$ on $G'$ to be the constant function $1$. Then by Theorem \ref{thm:classical max min} (or more precisely, the second paragraph below the theorem), the maximum number of edge disjoint directed paths $\MF(G';c)$ from $S$ to $T$ is equal to the minimum capacity $\MC(G';c)$ of edge cut sets. It is clear that $\MC(G)$ $= $ $\MC(G',c)$. We show below that $\MF(G)$ $\geq$ $\MF(G',c)$, which implies $\MF(G) \geq \MC(G)$.

Let $P$ denote the set of edge disjoint paths in $G'$ whose cardinality achieves the maximum number $\MF(G',c)$. Note that if the edges $(i,j),(j,i)$ both appear in a path $p \in P$, say $(u_1,u_2) \cdots (u_{k},i) (i,j) \cdots (j,i) (i,u_{r}) \cdots$, then we can just replace the path with a shorter one $(u_1,u_2) \cdots (u_{k},i)(i,u_{r}) \cdots$, which still connects an input to an output. Also note that if two paths in $P$ are of the form $p_1 = (u_1,u_2) \cdots (u_{k},i)(i,j)(j,u_{k+1}) \cdots , p_2 = (v_1,v_2) \cdots (v_{r},j)(j,i)(i,v_{r+1}) \cdots$, then we can replace them by $p_1' = (u_1,u_2) \cdots (u_{k},i)(i,v_{r+1}) \cdots , p_2' = (v_1,v_2) \cdots (v_{r},j)(j,u_{k+1}) \cdots$. By a sequence of the above operations, we can assume that for each pair of the edges $(i,j),(j,i)$ in $G'$, at most one of them appears in the collection of the edge disjoint paths $P$, and thus we can pick out the same number of edge disjoint paths in $G$. Therefore, $\MF(G)$ $\geq$ $\MF(G',c)$.

%To show $\MC(G)$ $ \leq $ $\MC(G',c)$, let $C'$ be an edge cut set whose capacity achieves the minimum $\MC(G',c)$. Since the capacity is $1$ on every edge, the capacity of $C'$ is simply the number of edges in $C'$. Let $C$ be the set of edges in $G$, in which each edge corresponds to some edge in $C'$ according to the rule defining the edges of $G'$. For instance, an edge $(i,j)$ in $G$ corresponds to the edges $(i,j)$ and $(j,i)$ in $G'$. Then $|C| \leq |C'| = \MC(G',c)$. The removal of $C$ disconnects $S$ from $T$, and thus $|C| \geq \MC(G)$, which implies $\MC(G) \leq \MC(G',c)$.

\end{proof}
\end{thm}

\section{Quantum max-flow min-cut theorem: Version I}
\label{sec:first}

We give a quantum analogue of the max-flow min-cut theorem where flow networks are replaced by tensor networks. The capacity on edges represents the dimension of a Hilbert space and the capacity of a tensor network thus behaves multiplicatively, instead of additively. After stating and proving the theorems, some additional context and applications will be given. Two versions of quantum $(\MF/\MC)$ will be provided in this section and next section, respectively.

Let $G = (\tilde{V},E)$ be a finite undirected graph with a set of inputs $S$ and a set of outputs $T$ such that $S \sqcup T$ is a disjoint partition of the set of degree $1$ vertices. Here it will be more convenient to assume $S$ and $T$ are not sets of vertices, but rather the open ends of some edges. Let $\tilde{V} = S \sqcup T \sqcup V$ be a partition of $\tilde{V}$. {\it Below, we will only call the elements in $V$ vertices}. For every vertex $v$ of degree $d_v$, we assume there is a local ordering $1, 2, \cdots, d_v$ of the ends of the edges incident to $v$. Let $e(v,1), e(v,2), \cdots, e(v,d_v)$ denote the edges incident to $v$ listed according to the local ordering. For each $u \in S \sqcup T,$ denote by $e(u)$ the edge whose open end is $u$. These edges are called input edges and output edges, respectively.

\begin{figure}
\begin{tikzpicture}[scale = 0.5]
\draw (0,0) -- (9,0) ;
\draw (0,3) -- (9,3) ;
\draw (0,1.5) -- (3,1.5) -- (6,3) ;
\draw (3,1.5) -- (6,0) ;

\draw (3,1.5) node{$\T_1$};
\draw (6,3) node{$\T_2$};
\draw (6,0) node{$\T_3$};

\draw (1,0) node{$2$};
\draw (1,1.5) node{$2$};
\draw (1,3) node{$2$};
\draw (8,0) node{$3$};
\draw (8,3) node{$3$};
\draw (4.5,0.75) node{$2$};
\draw (4.5,2.25) node{$2$};

\draw (-1,1.5) node{{\small{Input}}};
\draw (10,1.5) node{{\small{Output}}};
\end{tikzpicture}
\caption{A tensor network. The integer on each edge is the capacity of that edge. There are five open ends, three of which on the left form the input $S$ and the other two the output $T$. There are three vertices and they are assigned the tensor $\T_1, \T_2, \T_3$, respectively.}\label{fig:rk7 copy}
\end{figure}
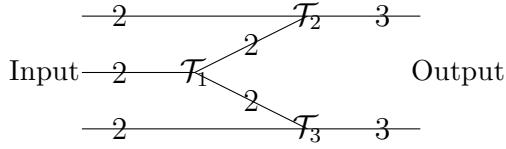

The graph $G$ now is a template for a tensor network. See Figure \ref{fig:rk7 copy}. To each edge $e$, we associate a Hilbert space $\C^{c_e}$, where $c: E \rightarrow \N, \, e \mapsto c_e$ acts as the quantum capacity. We fix a basis of $\C^{c_e}$, so that it allows us to freely raise and lower indices of the tensors to be introduced below. Now any assignment $v \mapsto \T_v$ taking each vertex $v$ to a tensor $\T_v \in \I^v:= \bigotimes\limits_{i=1}^{d_v} \C^{c_{e(v,i)}}$, which as totality can be written as $V \mapsto \T \in \I:= \bigoplus\limits_{v \in V} \I^v$, sends the graph $G$ to a tensor network, $G \mapsto N(G,c;\T)$. The linear ordering $1,2,\cdots, d_v$ specifies which index of $\T_v$ is identified with which edge end at $v$. As usual, graphical edges are interpreted as contraction of indices. Thus $N(G,c;\T)$ in turn determines an element $\alpha(G,c;\T) \in V_S \otimes V_T$, where $V_S := \otimes_{u \in S} \C^{c_{e(u)}} $ and $V_T = \otimes_{u \in T} \C^{c_{e(u)}} $. Using the standard basis in $\C^n$, this also determines an element $\beta(G,c;\T) \in V_S^{*} \otimes V_T = Hom(V_S, V_T)$. In a compact multi-index notation, we can write it as

$$\langle I_T | \beta(G,c;\T)|I_S\rangle = \sum\limits_{W \textrm{ extending } I_S,I_T} \prod \limits_{v \in V} \langle W_{|_v}|\T_v\rangle, $$
where $I_S$ and $I_T$ are multi-indices for input edges and output edges, $W$ is a multi-index for all edges and $W_{|_v}$ indicates the portion of $W$ which can be read at $v$.

With the notations above,
\begin{definition}
The quantum max-flow, $\QMF(G,c)$, is the the maximum rank of $\beta(G,c;\T)$ over all tensor assignments $\T$.
\end{definition}

An edge cut set $C$ is defined in the same way as in the classical case in Section \ref{sec:classical}. Namely, $C$ is a set of edges such that there exists a partition $\tilde{V} = \bar{S} \sqcup \bar{T}$ such that $S \subset \bar{S},\ T \subset \bar{T}$, and $C = \{(u,v) \in E: u \in \bar{S}, v \in \bar{T} \}$.

\begin{definition}
The quantum min-cut, $\QMC(G,c)$, is the minimum of $\prod \limits_{e \in C} c_e$, the quantum capacity, over all edge cut sets $C$.
\end{definition}

%The amount of \lq\lq quantum flow'' of $\T$ is defined to be the rank of $\beta(G,c;\T)$. Define $\QMF(G,c)$, the quantum max-flow, to be the maximum rank of $\beta(G,c;\T)$ over all tensor assignments $\T$. Let $C$ be an edge cut set which separates $S$ from $T$. Then the \lq\lq quantum capacity'' of $C$ is defined to be $n_C = \prod \limits_{e \in C} c_e$, and we let $\QMC(G,c)$, quantum min-cut, be the minimum of $n_C$ over all edge cut sets.

\begin{rem}\label{rem:ordering}
Although $N(G,c;\T),\alpha(G,c;\T) $ and $ \beta(G,c;\T)$ depend on the local ordering of the ends of edges incident to a vertex, we did not indicate this dependence in their notation to avoid prolixity. The quantity $\QMF(G,c)$, on the other hand, does not depend on the local ordering, since the tensors assigned to each vertex can be varied arbitrarily. By definition $\QMC(G,c)$ is also independent of the ordering. We will see in Section \ref{sec:second} that Version II of tensor networks depends on the local ordering in a critical way.
\end{rem}
\begin{rem}\label{rem:complexity}
$\QMC(G,c)$ can be calculated efficiently (that is, in polynomial time in the input's length) by running the efficient classical min cut algorithm with capacities $\log(c_1),\ldots,\log(c_n)$, and taking the exponent of the result. We do not know whether $\QMF(G,c)$ can be calculated efficiently.
\end{rem}

\begin{rem}
\label{re:replace input output} In a classical directed flow network, the max flow can change when the roles of the input $S$ and the output $T$ are replaced, but the max flow remains fixed when the flow network is undirected. This is also the case for quantum max flow: the roles of the inputs and the outputs can be interchanged without changing the quantum max flow; this follows from the equality of the dimensions of the image and the coimage.
\end{rem}

%We will show below that  $\QMF(G,c)$ and $\QMC(G,c)$ are only equal for some graphs with appropriate capacity functions, and for other graphs, they are not.

The functorial nature of tensor network immediately implies the following lemma.

\begin{lem}
\label{lem:functorial}
Let $G = (\tilde{V},E), \ S, \ T, \ c,\ \T$ be as above, and let $V_C = \bigotimes\limits_{e \in C} \C^{c_e}$, where $C$ is an edge cut set. Then $\beta(G,c;\T) \in Hom(V_S,V_T)$ factors as $\beta_2 \circ \beta_1$, where $\beta_1 \in Hom(V_S, V_C)$ and $\beta_2 \in Hom(V_C, V_T)$.
\begin{proof}
Let $\tilde{V} = \bar{S} \sqcup \bar{T}$ be the partition of $\tilde{V}$ such that, $S \subset \bar{S}, T \subset \bar{T}$, and the set of edges between $\bar{S}$ and $\bar{T}$ form the cut set $C$. Delete an interior point on each edge in $C$ so that each edge is split into two edges, each with one open end. Let $M$ be the set of the deleted points and let $G_1$ $($resp. $G_2$$)$ be the components of $G$, which contain $\bar{S}$ $($resp. $\bar{T}$$)$. Then $G_1$ is a graph with input $S$ and output $M$, and $G_2$ is a graph with input $M$ and output $T$. Also let $c_i$ and $\T_i$ be the restriction of $c$ and $ \T$ on $G_i$, respectively, $i=1,2$. Then it follows that $\beta(G,c;\T) = \beta(G_2,c_2; \T_2) \circ \beta(G_1,c_1;\T_1),$ and $\beta(G_1,c_1;\T_1) \in Hom(V_S,V_C), \beta(G_2,c_2;\T_2) \in Hom(V_C, V_T)$.
\end{proof}
\end{lem}

The Corollary below shows a basic property of tensor networks that any cut of the network provides an upper bound for the maximal rank, which is well known in possibly different forms in literature. For instance, if one views the map $\beta(G,c;\T)$ as an unnormalized state in $\in V_S \otimes V_T$, namely as $\alpha(G,c;\T)$, then the entropy of entanglement between $S$ and $T$ is upper bounded by the logarithm of the rank, and thus the entropy is also upper bounded by the logarithm of the min-cut which is just an upper bound version of the area law \cite{orus2014}. See also Lemma \ref{lem:ee<max}.

\begin{cor}
\label{lem:max<min}
Given a finite graph $G$ with the quantum capacity function $c$, then $\QMF(G,c)$ $\leq $ $\QMC(G,c)$.
\begin{proof}
For any tensor assignment $\T$ and any edge cut set $C$, by Lemma \ref{lem:functorial}, $\beta(G,c;\T) = \beta_2 \circ \beta_1$, where $\beta_1 \in Hom(V_S, V_C)$. Thus $rank(\beta(G,c;\T)) \leq rank(\beta_1) \leq dim(V_C)$.
\end{proof}
\end{cor}

In general, we do not know a necessary and sufficient condition for the inequality in Corollary \ref{lem:max<min} to become an equality. Theorem \ref{thm:max=min}, however, states that if the quantum capacity at all edges has a uniform tensor product structure, then the quantum version of max-flow min-cut theorem holds. The authors in \cite{pastawski2015} had a similar result (Theorem $2$ in \cite{pastawski2015}) for a more restricted class of tensor networks, where tensors are all perfect tensors, the quantum capacity is the same on all edges, and the underlying graph is required to have \lq\lq non-positive curvature''. Our result does not have any requirement on tensors and the shape of the graph.

\begin{thm}
\label{thm:max=min}
$[$Quantum Max-flow Min-cut Theorem$]\\$
Let $G = (\tilde{V},E)$ be as above. If there exists an integer $d > 0$, such that the capacity function $c$ at each edge is a power of $d$, then $\QMF(G,c)$ $= $ $\QMC(G,c)$.
\begin{proof}

For each edge $e$, let $m_e = \log_d c_e$. Let $G' = (\tilde{V}, \varnothing)$ be a graph with no edges. Then for each edge $e = (u,v)$ in $G$ with capacity $c_e$, if both $u$ and $v$ are vertices, we insert $m_e$ parallel edges connecting $u$ to $v$ in $G'$. If $u($resp. $v)$ is an open end, we insert in $G'$ $m_e$ open edges all incident to $v($resp. $u)$. Denote the resulting graph still by $G'$, and define the capacity function $c'$ to have value $d$ on each edge of $G'$. It can be seen that there is a one-to-one correspondence between the tensor assignments in $G$ and tensor assignments in $G'$ such that the the resulting linear maps have the same rank. Moreover, a direct consequence of the definition of $G'$ is that $\QMC(G,c)$ = $\QMC(G',c')$. Therefore, if the theorem holds for $(G',c')$, it also holds for $(G,c)$. Thus, without loss of generality, we assume that the capacity of each edge in $G$ is equal to $d$.

%Let $M =$ $\QMC(G,c)$. Let $G''$ be the same graph as $G$, but with each edge assigned the capacity $1$. Treated as a classical network, $G''$ has a min-cut equal to $\log_d M$. Then by Theorem \ref{thm:Menger}, the number of edge disjoint paths in $G''$ (hence also in $G$) from the sources to the sinks is $\log_d M$. For each vertex $v$ in $G$, assign a tensor $\T_v$ as follows.

Let $M =$ $\QMC(G,c)$. We give an explicit tensor assignments so that the resulting linear map has rank equal to $M$. Viewed as a classical network, there are $\log_d M$ edge disjoint paths in $G$ from the sources to the sinks by Theorem \ref{thm:Menger}. Denote these paths by $p_1, p_2, \cdots$. For each vertex $v$ in $G$, the tensor $\T_v$ is assigned $1$ if and only if the following rules are satisfied for the indices of edges incident to $v$:
\begin{enumerate}
 \item If two edges are adjacent on one of the paths $p_i\,'$s, then they have the same index.
 \item An edge which does not belong to any $p_i\,'$s has index $1$, (the first index in the index set $\{1,\cdots, d\}$).
\end{enumerate}

\noindent $T_v$ is assigned $0$ otherwise. It is clear that for the contraction of the tensor network to be non-zero, which must be $1$ actually, the indices on all edges of a path $p_i$ must be the same, and indices on the edges which do not belong to any $p_i\,'$s must be $1$. Thus there are in total $d^{\log_d (M)}$ configurations of indices which make the contraction equal $1$. Therefore after an appropriate ordering of the basis elements in $V_S$ and in $V_T$, respectively, the map $\beta(G,c;\T)$ is of the form shown in Equation \ref{mat:beta}, where the dimension of the upper left block is $M \times M$. Hence, $\QMF(G,c)$ $\geq$  $\QMC(G,c)$. Combining with Corollary \ref{lem:max<min}, we have $\QMF(G,c)$ $=$  $\QMC(G,c)$.

\begin{equation}\label{mat:beta}
\left(
\begin{array}{ccc|ccc}
1 &             &   & \multicolumn{3}{c}{\multirow{3}{*}{\huge{0}}}\\
  & \ddots      &   & \multicolumn{3}{c}{\multirow{3}{*}{}} \\
  &             & 1 & \multicolumn{3}{c}{\multirow{3}{*}{}} \\  \hline
\multicolumn{3}{c}{\multirow{3}{*}{\huge{0}}} \vline& \multicolumn{3}{c}{\multirow{3}{*}{\huge{0}}} \\
\multicolumn{3}{c}{\multirow{3}{*}{}}         \vline& \multicolumn{3}{c}{\multirow{3}{*}{}}\\
\multicolumn{3}{c}{\multirow{3}{*}{}}         \vline& \multicolumn{3}{c}{\multirow{3}{*}{}}\\
\end{array}
\right)
\end{equation}

\end{proof}
\end{thm}

To summarize, we used the classical solution to find an exceedingly simple list of tensors $\T$, all entries of which are $0$ or $1$, which provide a solution instance to the quantum problem. However, the following proposition shows that as long as one can find one solution, almost all choices of tensors $\T$ are also solutions.

\begin{prop}
\label{prop:open dense}
Let $G = (\tilde{V},E)$ be as above with the capacity function $c$. Then the set of all tensors $\T \in \I:= \bigoplus\limits_{v \in V} \I^v$ such that $rank(\beta(G,c;\T))$ is equal to $\QMF(G,c)$, is an open dense subset of $\I$.
\begin{proof}
Let $M = $ $\QMF(G,c)$. Consider the general assignment $\T \in \I$ of tensors to vertices $\T = \{\T_v: v \in V\}$. The condition that $rank(\beta(G,c;\T))< M$ is equivalent to the vanishing condition of a set of polynomials $\{P_{\alpha}\}$, each of which is the determinant of some $M \times M$ minor of $\beta(G,c;\T)$. Thus the set of all tensors $\T$ which give $rank(\beta(G,c;\T)) < M$ is a proper affine algebraic variety, the complement of which is an open dense subset of $\I$ \cite{hartshorne1977algebraic}.

 %$\mathcal{V} = \{\T \in \I: P_{\alpha}(\T) = 0, \forall \alpha\}$. Since there is at least one $\T_0 \not\in \mathcal{V}$, so $\mathcal{V} $ is a proper subvariety of $\I$. This immediately implies \cite{hartshorne1977algebraic} that $\mathcal{V}$ is at least complex codimension $1$, the Lebesgue measure of $\mathcal{V}$ is $0$ and $\I \setminus \mathcal{V}$ is open dense.
\end{proof}
\end{prop}

\begin{prop}
Let $G = (\tilde{V},E)$ be a tensor network with the capacity function $c:E \rightarrow \N$, and let $G' = (\tilde{V},E)$ be the classical network with the same underlying graph as $G$, but with the capacity function $c' = \log_d c: E \rightarrow \R^{+}$, for some $d > 0$. Direct $G'$ such that $G'$ has the same set of inputs and outputs as $G$, and $\MC(G',c') = \log_d \QMC(G,c)$. If $G'$ has a max flow with the flow amount at each edge being the $\log_d$ of some integer, and there is no loop in $G'$ such that the flow on each edge of the loop is non-zero, then $\QMF(G,c)=\QMC(G,c)$.

\begin{proof}
Choose a max flow $f: E \rightarrow \R^{+}$ for $G'$ as stated in the proposition. Consider the tensor network $G'' = (\tilde{V},E)$ with the capacity function $c'' = d^f: E \rightarrow \N.$ $G''$ has the same underline graph as $G'$, and in particular, $G''$ is directed. Then since $c''(e) = d^{f(e)} \leq d^{c'(e)} = c(e)$, we have $\QMF(G'',c'')\leq \QMF(G,c)$. On the other hand, $\QMC(G'',c'') = d^{\MC(G',c')} = \QMC(G,c)$. Hence to prove $\QMF(G,c) = \QMC(G,c)$, it suffices to show $\QMF(G'',c'') = \QMC(G'',c'')$.

By construction, $G''$ has the property that at each vertex, the product of the capacities on the incoming edges is equal to that of the capacities on the outgoing edges. Moreover, the set of edges associated to the sources is a min-cut set. Thus, to each vertex one can associate a linear isomorphism from the space of the incoming edges to the space of the outgoing edges. Since there is no loop in $G''$, the resulting linear map of the tensor network is simply a composition of linear isomorphisms, and hence is an isomorphism from the input space to the output space. Therefore, $\QMF(G'',c'') = \QMC(G'',c'')$.
\end{proof}
\end{prop}

\begin{prop}\label{prop:lower_bound}[Lower bound on capacity.]
If a graph $G$ has edges of capacity $d_\text{min} \leq \cdots \leq d_\text{max}$ and a classical min cut of cardinality $C$, then the quantum capacity satisfies the lower bound: $\QMF(G, c) \geq d^C_\text{min}$.
\end{prop}

\begin{proof}
We may always restrict to a subspace $V$ of any edge space and correspondingly restrict to tensors which vanish when input at any index is from the orthogonal complement $V^\bot$.  Doing this and applying Theorem \ref{thm:max=min} give the result.
\end{proof}

\begin{note}
A somewhat better lower bound can be obtained by \lq\lq thinning'' $G$ to $\widetilde{G}$ by reducing, for each edge $e$,  the capacity (i.e. dimension) $d_e$ to the largest power $d_\text{min}^{\phi_e} \leq d_e$ and then computing the $\QMF(\widetilde{G}, c) = \QMC(\widetilde{G}, c)$ of the thinned graph $\widetilde{G}$.
\end{note}

Before presenting any examples, we need a technical lemma where one can found the proof in \cite{dur2000three} or an independent proof in Appendix \ref{sec:proof lemma dense}.

\begin{lem}\cite{dur2000three}\label{lem:map dense}
Let $U,V,W$ be vector spaces isomorphic to $\C^2$. Then the set of linear maps $\Phi: U \longrightarrow V \otimes W$, which can be written in the form $|1\rangle_U \mapsto |1\rangle_V\otimes|1\rangle_W$, $|2\rangle_U \mapsto |2\rangle_V\otimes|2\rangle_W$ under appropriate bases of $U,V$ and $W$, is an open dense subset of $Hom(U, V \otimes W)$.
\end{lem}

Let $\Sen = \{\Sen_{ijk}: i,j,k = 0, 1\} \in (\C^2)^{\otimes 3}$ be the tensor such that $\Sen_{ijk} = 1$ if $i=j=k$, and $\Sen_{ijk} = 0$ otherwise. Translating Lemma \ref{lem:map dense} into the language of tensors, we have the following corollary.
\begin{cor} \label{cor:map dense}
The set of tensors $\T = \{\T_{ijk}: i,j,k = 0, 1\} \in (\C^2)^{\otimes 3}$, which satisfies the property that there exist invertible tensors $A = \{A_{ij}:i,j = 0,1\}$, $B = \{B_{ij}:i,j = 0,1\}$, $C = \{C_{ij}:i,j = 0,1\}$ such that $\T_{ijk} = \sum\limits_{a,b,c} A_{ia}\Sen_{abc}B_{bj}C_{ck}$ or graphically the equality in Figure \ref{tensorTS1} holds, is a dense subset of $(\C^2)^{\otimes 3}$.
\end{cor}

\begin{figure}
\begin{tikzpicture}[scale = 0.5]
\begin{scope}
\draw (0,0) -- (2,0) -- (4,2);
\draw (2,0) -- (4,-2);
\draw (2,0.8) node{$\T$};
\draw (1.7,0) node{{\tiny{$1$}}};
\draw (2.3,0.3) node{{\tiny{$2$}}};
\draw (2.3,-0.3) node{{\tiny{$3$}}};
\filldraw (2,0) circle(3pt);
\end{scope}

\draw (6,0) node{$=$};

\begin{scope}[xshift = 8cm]
\draw (0,0) -- (3,0) -- (6,3);
\draw (3,0) -- (6,-3);
\draw (3,0.8) node{$\Sen$};
\draw (1.5,0.8) node{$A$};
\draw (4.5,2.3) node{$B$};
\draw (4.5,-2.3) node{$C$};

\filldraw (3,0) circle(3pt);
\filldraw (1.5,0) circle(3pt);
\filldraw (4.5,1.5) circle(3pt);
\filldraw (4.5,-1.5) circle(3pt);

\draw (2.7,0) node{{\tiny{$1$}}};
\draw (3.3,0.3) node{{\tiny{$2$}}};
\draw (3.3,-0.3) node{{\tiny{$3$}}};

\draw (1.2,0) node{{\tiny{$1$}}};
\draw (1.8,0) node{{\tiny{$2$}}};

\draw (4.2,1.2) node{{\tiny{$1$}}};
\draw (4.8,1.8) node{{\tiny{$2$}}};

\draw (4.2,-1.2) node{{\tiny{$1$}}};
\draw (4.8,-1.8) node{{\tiny{$2$}}};

\end{scope}
\end{tikzpicture}
\caption{Tensors $\T$ and $\Sen$} \label{tensorTS1}
\end{figure}
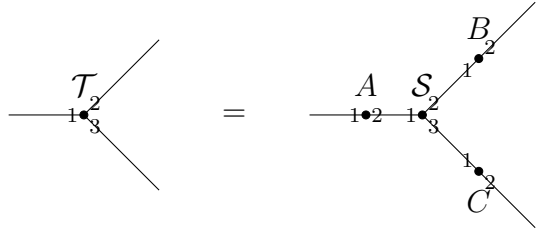

Theorem \ref{thm:max=min} shows that if the capacity of each edge in a graph is a power of some fixed integer, then the quantum max-flow equals  quantum min-cut. However, in general this equality may or may not hold. See Example \ref{eg:rk7 8} \ref{eg:rk 2n^2-jk} for the illustrations. By Remark \ref{rem:ordering}, the local ordering of the edges around a vertex does not matter. We will put an integer on each edge to represent the quantum capacity, or dimension of the local Hilbert space.

\begin{example} \label{eg:rk7 8}
 Let $G_1$ be the graph shown in Figure \ref{fig:rk7} with capacity function $c_1$. This is a special case of the network in Figure \ref{fig:rk 2n^2-jk}, Example \ref{eg:rk 2n^2-jk} with $n=2, j=k=1$. So by the conclusion there, $\QMC(G_1,c_1) = 8$ and $\QMF(G_1,c_1) \leq 7$, and thus the quantum max-flow min-cut theorem does not hold. On the other hand, in Figure \ref{fig:rk8}, if we use the same graph as that in Figure \ref{fig:rk7}, but change the capacity of the two internal edges to $p$ and $q$ with $p, q \geq 2$, then direct calculations show that $\QMF(G_1,c_2)$ equals $8$ as long as $p \geq 3$ or $q \geq 3$, in which case the quantum max-flow min-cut theorem holds.
\begin{figure}
\begin{tikzpicture}[scale = 0.5]
\draw (0,0) -- (9,0) ;
\draw (0,3) -- (9,3) ;
\draw (0,1.5) -- (3,1.5) -- (6,3) ;
\draw (3,1.5) -- (6,0) ;

\draw (1,0) node{$2$};
\draw (1,1.5) node{$2$};
\draw (1,3) node{$2$};
\draw (8,0) node{$3$};
\draw (8,3) node{$3$};
\draw (4.5,0.75) node{$2$};
\draw (4.5,2.25) node{$2$};

\draw (-1,1.5) node{{\small{Input}}};
\draw (10,1.5) node{{\small{Output}}};
\end{tikzpicture}
\caption{$(G_1,c_1)$}\label{fig:rk7}
\end{figure}

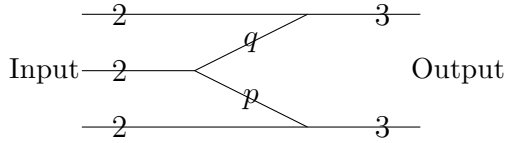
\begin{figure}
\begin{tikzpicture}[scale = 0.5]
\draw (0,0) -- (9,0) ;
\draw (0,3) -- (9,3) ;
\draw (0,1.5) -- (3,1.5) -- (6,3) ;
\draw (3,1.5) -- (6,0) ;

\draw (1,0) node{$2$};
\draw (1,1.5) node{$2$};
\draw (1,3) node{$2$};
\draw (8,0) node{$3$};
\draw (8,3) node{$3$};
\draw (4.5,0.75) node{$p$};
\draw (4.5,2.25) node{$q$};

\draw (-1,1.5) node{{\small{Input}}};
\draw (10,1.5) node{{\small{Output}}};
\end{tikzpicture}
\caption{$(G_1,c_2)$, $p, q \geq 2$}\label{fig:rk8}
\end{figure}
\end{example}

\begin{example} \label{eg:rk 2n^2-jk}
With the same graph $G_1$, one can generalize the quantum capacity function in another direction, as shown in Figure \ref{fig:rk 2n^2-jk}, where the capacity function is denoted by $c_3$, $\T_1, \T_2, \T_3$ are three tensors assigned to each vertex, and we assume $n \geq 2, j,k < n$. For the network in Figure \ref{fig:rk 2n^2-jk}, we have $\QMC(G_1,c_3) = \min \{2n^2, (2n-j)(2n-k)\}$. In the following, we prove that $\QMF(G_1,c_3) \leq 2n^2-jk$, which is strictly less than $\QMC(G_1,c_3)$.

\begin{note}
Optimizing example \ref{eg:rk 2n^2-jk}, we find, for $j=k\approx \left(2 - \sqrt{2}\right)n$, networks where the quantum/classical ratio $\frac{\QMF(G,c)} {\exp\left(\sum_\text{min cut} \ln(\text{cut dimension})\right)}$ approaches $2\left(\sqrt{2} - 1\right) \approx 0.8284271$.  This is the smallest ratio, i.e. greatest discrepancy from the analog of classical capacity (the denominator) that we have so far obtained with $3$-input networks.  Of course, with more inputs, this ratio may be driven to zero: $p$ parallel copies of such a network will have quantum/classical ratio approaching $\left[2\left(\sqrt{2} - 1\right)\right]^p$.
\end{note}

 One can view $\T_1, \T_2, \T_3$ as linear maps $\C^2 \longrightarrow \C^2 \otimes \C^2, \, \C^n \otimes \C^2 \longrightarrow  \C^{2n-k},\, \C^2 \otimes \C^n \longrightarrow  \C^{2n-j},$ respectively. By Lemma \ref{lem:map dense}, a generic $\T_1$, via a local change of basis, can be transformed to the map: $|0\rangle \mapsto |00\rangle, |1\rangle \mapsto |11\rangle$. Since a local change of basis does not affect the rank of $\beta(G_1,c_3;\T)$, we can assume $\T_1$ is given by the map just mentioned.

For a generic choice of $\T_2: \C^n \otimes \C^2 \longrightarrow  \C^{2n-k}$ and $i=0,1$, we can view $\phi_i:= \T_2(\cdot \otimes |i\rangle)$ as a linear map $\C^n \longrightarrow \C^{2n-k}$. Then generically, both the image of $\phi_0$ and the image of $\phi_1$ are of dimension $n$, and  thus they must have an overlap of dimension at least $k$. So there exist two sets of linearly independent vectors, $\{v_1, \cdots, v_k\}$, $\{u_1, \cdots, u_k\}$, such that $\T_2(v_i \otimes |0\rangle) = \T_2(u_i \otimes |1\rangle), 1 \leq i \leq k$. By the same argument, there exist two sets of linearly independent vectors, $\{w_1, \cdots, w_j\}$, $\{x_1, \cdots, x_j\}$ such that $\T_3(|0\rangle \otimes w_h) = \T_2(|1\rangle \otimes x_h), 1 \leq h \leq j.$ Then the subspace spanned by $\{v_i \otimes |00\rangle \otimes w_h - u_i \otimes |11\rangle \otimes x_h: 1 \leq i \leq k, 1 \leq h \leq j \}$ with dimension $jk$ is contained in the kernel of $\T_2\otimes \T_3$. Note that this subspace is also contained in the image of $Id_{\C^n} \otimes \T_1 \otimes Id_{\C^n}$, therefore the image of $\beta(G_1,c_3; \T) = (\T_1\otimes \T_2)\circ(Id_{\C^n} \otimes \T_1 \otimes Id_{\C^n})$ is at most $2n^2 - jk$, which implies that the rank of $\beta(G_1,c_3;\T)$ is at most $2n^2-jk$. By Proposition \ref{prop:open dense}, the set of tensors which realize the maximum rank is an open dense subset, hence $\QMF(G_1,c_3) \leq 2n^2-jk$.

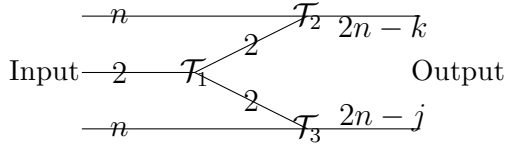
\begin{figure}
\begin{tikzpicture}[scale = 0.5]
\draw (0,0) -- (9,0) ;
\draw (0,3) -- (9,3) ;
\draw (0,1.5) -- (3,1.5) -- (6,3) ;
\draw (3,1.5) -- (6,0) ;

\draw (3,1.5) node{$\T_1$};
\draw (6,3) node{$\T_2$};
\draw (6,0) node{$\T_3$};

\draw (1,0) node{$n$};
\draw (1,1.5) node{$2$};
\draw (1,3) node{$n$};
\draw (8,0.3) node{$2n-j$};
\draw (8,2.7) node{$2n-k$};
\draw (4.5,0.75) node{$2$};
\draw (4.5,2.25) node{$2$};

\draw (-1,1.5) node{{\small{Input}}};
\draw (10,1.5) node{{\small{Output}}};
\end{tikzpicture}
\caption{$(G_1,c_3), n \geq 2, j,k < n$}\label{fig:rk 2n^2-jk}
\end{figure}
\end{example}

\section{Quantum max-flow min-cut theorem: Version II}
\label{sec:second}

Here we study a second version of quantum max-flow min-cut theorem for a more restricted class of tensor networks originally motivated by \cite{calegari2010}. Roughly speaking, vertices of the same type (to be defined below) in a tensor network are required to be assigned the same tensor. In \cite{calegari2010}, it was conjectured that the quantum max-flow for this version equals the quantum min-cut. And if it were true, it implies the existence of some tensors with certain \lq\lq positive'' properties, which is part of the work in \cite{calegari2010} for proving the positivity of the universal pairing in dimension $3$. Unfortunately, we show below that this conjecture is false by a concrete example. \footnote{Nevertheless, the authors in \cite{calegari2010} avoided the use of the conjecture and constructed the tensors with more efforts.} Some more examples and properties will also be presented.

Notations from Section \ref{sec:first} will be used here. For $G = (\tilde{V},E)$ with a capacity function $c$ and a local ordering $L$ of the ends of the edges incident to each vertex, we define the {\it{valence type}} $B_v$ of a vertex $v$ to be the sequence $(c_{e(v,1)}, \cdots, c_{e(v,d_v)})$, and define $\B(G,c,L)$ to be the set of valence types of vertices of $G$. Let $\I^B = \bigotimes\limits_{i=1}^k \C^{m_i}$ for a {\it{valence type}} $B = (m_1, \cdots, m_k)$, and let $\I' = \bigoplus\limits_{B \in \B(G,c,L)} \I^B$. {\it Now the vertices with the same valence type have to be assigned the same tensor}. Given a family of tensors $\T = \{\T_B: B \in \B(G,c,L)\} \in \I'$, a vertex $v$ with valence type $B_v$ is assigned the tensor $\T_{B_v}$ $($according to the local ordering of its incident edges$)$. Again contracting the graphical edges of the tensor network results in a linear map, denoted by $\beta(G,c,L;\T)$, in $Hom(V_S,V_T)$.

\begin{definition}
The quantum max-flow, $\QMF(G,c,L)$, for Version II is defined to be the maximum rank of $\beta(G,c,L;\T)$.
\end{definition}

The quantum min-cut $\QMC(G,c)$ is the same in either version. It is clear from the definitions that $\QMF(G,c,L)$ $\leq$ $\QMF(G,c)$, and thus $\QMF(G,c,L) \leq \QMC(G,c)$. When does the equality hold? Again, we do not know a sufficient and necessary condition to this question. let us first look at some examples.

%Define $\QMF(G,c,L)$ to be the maximum rank of $\beta'(G,c,L;\T)$. We still denote by $\QMC(G,c)$ the minimum quantum capacity over all edge cut sets. It is clear that $\QMF(G,c,L)$ $\leq$ $\QMF(G,c)$.

%From the observation above and Corollary \ref{lem:max<min}, we have $\QMF(G,c,L) \leq \QMC(G,c)$. Then the natural question is: when does the equality hold? Let's first look at some examples.

All edges of the graphs in Example \ref{eg:rk3 4}, \ref{eg:rk 4} and \ref{eg:rk6} have capacity $2$, so we omit the labels of the capacity to make the pictures nicer. Instead, we put a number at each edge end incident to a vertex to stand for the local ordering. Note that in this version, the local ordering affects the valence type, and so it also affects maximal rank of the tensor network.

\begin{example} \label{eg:rk3 4}
In Figure \ref{fig:rk3}, denote the graph, the capacity function and the local ordering by $G_1,c_1$ and $L_1$, respectively. There is only one valence type of the vertices, namely $(2,2,2)$. So we only need to choose one tensor $\T \in (\C^2)^{\otimes 3}$, and assign it to each vertex according to the local ordering $L_1$. Clearly $\QMC(G_1,c_1)$ is equal to $4$. However, it will be proved in Appendix \ref{subsec:example rk 3} that $\QMF(G_1,c_1,L_1) \leq 3$. Thus $\QMF(G_1,c_1,L_1) < \QMC(G_1,c_1)$.

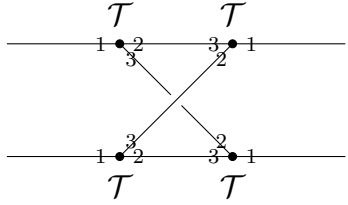
\begin{figure}
\begin{tikzpicture}[scale = 0.5]
\begin{scope}
\draw (3,3) -- (6,0);
\draw [color = white, line width = 2mm](3,0) -- (6,3);
\draw (3,0) -- (6,3);
\draw (0,3) -- (9,3);
\draw (0,0) -- (9,0) ;

\filldraw (3,0) circle(3pt);
\filldraw (3,3) circle(3pt);
\filldraw (6,0) circle(3pt);
\filldraw (6,3) circle(3pt);

\draw (3,-0.8) node{$\T$};
\draw (3,3.8) node{$\T$};
\draw (6,-0.8) node{$\T$};
\draw (6,3.8) node{$\T$};

\draw (2.5,0) node{{\tiny{$1$}}};
\draw (3.5,0) node{{\tiny{$2$}}};
\draw (3.3,0.4) node{{\tiny{$3$}}};

\draw (5.5,0) node{{\tiny{$3$}}};
\draw (6.5,0) node{{\tiny{$1$}}};
\draw (5.7,0.4) node{{\tiny{$2$}}};

\draw (2.5,3) node{{\tiny{$1$}}};
\draw (3.5,3) node{{\tiny{$2$}}};
\draw (3.3,2.6) node{{\tiny{$3$}}};

\draw (5.5,3) node{{\tiny{$3$}}};
\draw (6.5,3) node{{\tiny{$1$}}};
\draw (5.7,2.6) node{{\tiny{$2$}}};

\end{scope}
\end{tikzpicture}
\caption{$(G_1,c_1,L_1)$}\label{fig:rk3}
\end{figure}

\end{example}

\begin{example} \label{eg:rk 4}
This example will show that the ordering on the ends of edges is crucial. In Figure \ref{fig:rk4}, the same graph as that in Figure \ref{fig:rk3} is drawn with the same local ordering for all vertices except the one on the lower right. We denote this new ordering by $L_2$. Then it turns out that $\QMF(G_1,c_1,L_2) = 4 = \QMC(G_1,c_1)$.

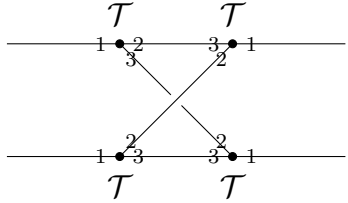
\begin{figure}
\begin{tikzpicture}[scale = 0.5]
\begin{scope}
\draw (3,3) -- (6,0);
\draw [color = white, line width = 2mm](3,0) -- (6,3);
\draw (3,0) -- (6,3);
\draw (0,3) -- (9,3);
\draw (0,0) -- (9,0) ;

\filldraw (3,0) circle(3pt);
\filldraw (3,3) circle(3pt);
\filldraw (6,0) circle(3pt);
\filldraw (6,3) circle(3pt);

\draw (3,-0.8) node{$\T$};
\draw (3,3.8) node{$\T$};
\draw (6,-0.8) node{$\T$};
\draw (6,3.8) node{$\T$};

\draw (2.5,0) node{{\tiny{$1$}}};
\draw (3.5,0) node{{\tiny{$3$}}};
\draw (3.3,0.4) node{{\tiny{$2$}}};

\draw (5.5,0) node{{\tiny{$3$}}};
\draw (6.5,0) node{{\tiny{$1$}}};
\draw (5.7,0.4) node{{\tiny{$2$}}};

\draw (2.5,3) node{{\tiny{$1$}}};
\draw (3.5,3) node{{\tiny{$2$}}};
\draw (3.3,2.6) node{{\tiny{$3$}}};

\draw (5.5,3) node{{\tiny{$3$}}};
\draw (6.5,3) node{{\tiny{$1$}}};
\draw (5.7,2.6) node{{\tiny{$2$}}};

\end{scope}
\end{tikzpicture}
\caption{$(G_1,c_1,L_2)$}\label{fig:rk4}
\end{figure}
\end{example}

\begin{example} \label{eg:rk6}
Denote the graph, the capacity function and the local ordering by $G_2,c_2$ and $L_3$, respectively, in Figure \ref{fig:rk6}. Then by a similar proof as that in Example \ref{eg:rk3 4}, we have $\QMF(G_2,c_2,L_3) = 6 < 8 = \QMC(G_2,c_2)$. However, since the capacity on each edge is $2$, by Theorem \ref{thm:max=min}, the first version of quantum max-flow min-cut theorem holds: $\QMF(G_2,c_2) = \QMC(G_2,c_2) = 8$.
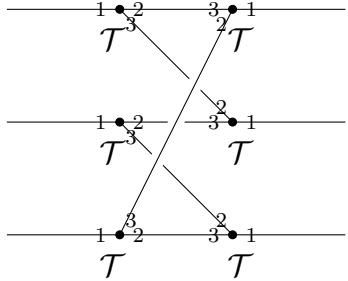
\begin{figure}
\begin{tikzpicture}[scale = 0.5]
\begin{scope}
\draw (3,3) -- (6,0);
\draw (3,6) -- (6,3);
\draw (0,3) -- (9,3);
\draw [color = white, line width = 2mm](3,0) -- (6,6);
\draw (3,0) -- (6,6);
\draw (0,0) -- (9,0);
\draw (0,6) -- (9,6);

\filldraw (3,0) circle(3pt);
\filldraw (3,3) circle(3pt);
\filldraw (3,6) circle(3pt);
\filldraw (6,0) circle(3pt);
\filldraw (6,3) circle(3pt);
\filldraw (6,6) circle(3pt);
%\draw (1,-0.5) node{$2$};
%\draw (4.5,-0.5) node{$2$};
%\draw (8,-0.5) node{$2$};
%
%\draw (1,3.5) node{$2$};
%\draw (4.5,3.5) node{$2$};
%\draw (8,3.5) node{$2$};
%
%\draw (3.6,2) node{$2$};
%\draw (3.6,1) node{$2$};
\draw (2.8,-0.8) node{$\T$};
\draw (2.8,2.2) node{$\T$};
\draw (6.2,-0.8) node{$\T$};
\draw (6.2,2.2) node{$\T$};
\draw (2.8,5.2) node{$\T$};
\draw (6.2,5.2) node{$\T$};

\draw (2.5,0) node{{\tiny{$1$}}};
\draw (3.5,0) node{{\tiny{$2$}}};
\draw (3.3,0.4) node{{\tiny{$3$}}};

\draw (5.5,0) node{{\tiny{$3$}}};
\draw (6.5,0) node{{\tiny{$1$}}};
\draw (5.7,0.4) node{{\tiny{$2$}}};

\draw (2.5,3) node{{\tiny{$1$}}};
\draw (3.5,3) node{{\tiny{$2$}}};
\draw (3.3,2.6) node{{\tiny{$3$}}};

\draw (5.5,3) node{{\tiny{$3$}}};
\draw (6.5,3) node{{\tiny{$1$}}};
\draw (5.7,3.4) node{{\tiny{$2$}}};

\draw (2.5,6) node{{\tiny{$1$}}};
\draw (3.5,6) node{{\tiny{$2$}}};
\draw (3.3,5.6) node{{\tiny{$3$}}};

\draw (5.5,6) node{{\tiny{$3$}}};
\draw (6.5,6) node{{\tiny{$1$}}};
\draw (5.7,5.6) node{{\tiny{$2$}}};

%\draw (-2,3) node{{\small{Input}}};
%\draw (12,3) node{{\small{Output}}};
\end{scope}
\end{tikzpicture}
\caption{$(G_2,c_2,L_3)$}\label{fig:rk6}
\end{figure}
\end{example}

The examples above showed that even in the case where all edges have the same capacity (dimension), the quantum max-flow/min-cut equality can still fail to hold, and the local orderings also affect the equality.

The following proposition states that one can, for each valence type $B$, fix a tensor $\T_B$, such that for all graphs $G$, when assigned the tensors $\{\T_B: B \in \B(G,c,L)\}$, the resulting $\beta(G,c,L;\T)$ has maximum rank.

\begin{prop}
Let $\B$ be the countable set of finite sequences of positive integers and let $\T^0 = \{\T^0_B: B \in \B\} \in \prod\limits_{B \in \B} \I^B$ have the property that the entries of all $\T^0_B\; '$s are algebraically independent from each other over the rational field $\Q$. Then for any graph $G$ with a capacity function $c$ and a local ordering $L$, the assignment $v \mapsto \T^0_{B_v},$ yields a linear map $\beta(G,c,L;\T^0_{|_{G}})$ of maximum rank $\QMF(G,c,L)$.
\begin{proof}
By the proof in Proposition \ref{prop:open dense}, for a graph $G$ with a general assignment $\T$ of tensors, the condition that $\beta(G,c,L;\T)$ does not have maximum rank is equivalent to the vanishing condition of a set of polynomial equations in the entries of $\T$. None of these polynomials are, for each $G$, identically $0$, since there are tensors which realize the maximum rank and thus those tensors do not satisfy these polynomial equations. Hence the value of these polynomials at $\T^0$ is not $0$ because the entries of all the $\T^0_{B}\; '$s are algebraically independent from each other, which implies for each $G$, the maximum rank is achieved.
\end{proof}
\end{prop}

\section{Entanglement Entropy}
\label{sec:entanglement}

In this section, we explore the quantum max-flow min-cut theorem in the context of quantum entanglement entropy, and we will only consider Version I of the tensor networks. Let's first review.

Let $\rho$ be the density matrix of a mixed state in a quantum system with the Hilbert space $\Hilb \cong \C^d$, then the Von Neumann entropy \cite{wehrl1978general}, $\Sen(\rho)$,  of $\rho$ is defined to be
\begin{equation}
\Sen(\rho):= -Tr (\rho \log \rho).
\end{equation}

Let $\lambda_i, i=1, \cdots, d$ be the eigenvalues of $\rho$. If we use the convention $0 \log 0 = 0,$ then we have
\begin{equation}
\Sen(\rho) = -\sum\limits_{i=1}^d \lambda_i \log \lambda_i.
\end{equation}

%The Von Neumann entropy measures the departure of a state from a pure state; $\Sen(\rho) \geq 0$, and \lq\lq $=$'' holds if and only if $\rho$ is a pure state; $\Sen(\rho) \leq \log d$, and \lq\lq $=$'' holds if and only if $\rho$ is a maximally mixed state, i.e. $\rho = \frac{1}{d}I_d$.

For a composite system with the Hilbert space $\Hilb = \Hilb_A \otimes \Hilb_B$, and a pure state $|\psi\rangle \in \Hilb$, the $($Von Neumann$)$ entanglement entropy, which we denote by $\EE(|\psi\rangle)$, of $|\psi\rangle$ is $\Sen(\rho_A(|\psi\rangle))$, where $\rho_A(|\psi\rangle)$ is the reduced density matrix of $|\psi\rangle$ on $A$. By Schmidt decomposition Theorem, $\rho_A(|\psi\rangle)$ has the same set of nonzero eigenvalues as $\rho_B(|\psi\rangle)$, thus $\EE(|\psi\rangle)$ is also equal to $\Sen(\rho_B(|\psi\rangle))$.

The entanglement entropy of a state in a bipartite system measures the entanglement between the subsystems. When $|\psi\rangle = |\psi\rangle_A \otimes |\psi\rangle_B$ is a product state, then $\EE(|\psi\rangle) = 0,$ which is the minimum value of $\EE$; When the reduced density matrix $\rho_A$ or $\rho_B$ is maximally mixed, $\EE(|\psi\rangle)$ achieves its maximum $\min\{\log d_A, \log d_B\} $, where $d_A, d_B$ are the dimensions of the subsystems $\Hilb_A, \Hilb_B$, respectively.

Recall from Section \ref{sec:first}, that for a finite graph $G= (\tilde{V},E)$ with input $S$, output $T$, the capacity function $c$ and a tensor assignment $T = \{\T_v: v \in V\}$, we have the associated linear map $\beta(G,c;\T) \in Hom(V_S, V_T) = V_S^* \otimes V_T$. Now identifying $V_S^*$ with $V_S$ using the chosen basis in $V_S$, this determines an element $\alpha(G,c;\T) \in V_S \otimes V_T$ (which is exactly the $\alpha(G,c;\T)$ introduced in Section \ref{sec:first}). More explicitly, if we denote the basis of $V_S, V_T$ by $\{|i\rangle_S: i =1,2, \cdots \}, \, \{|j\rangle_T:j=1,2, \cdots\}$, and let the matrix of $\beta(G,c;\T)$ under this basis be $C$, then we have:
\begin{equation}
\alpha(G,c;\T) = \sum\limits_{i,j} C_{ij} |i\rangle_S|j\rangle_T.
\end{equation}

Since $Tr(CC^{\dag}) = \sum\limits_{i,j} |C_{ij}|^2$, $\frac{\alpha(G,c;\T)}{\sqrt{Tr(CC^{\dag})}}$ is a normalized state in $ V_S \otimes V_T$. The reduced density matrix $\rho_S(\frac{\alpha(G,c;\T)}{\sqrt{Tr(CC^{\dag})}})$ equals $\frac{CC^{\dag}}{Tr(CC^{\dag})}$. Define the entanglement entropy between $S$ and $T$ for a given tensor assignment $\T$:
\begin{eqnarray}
\EE(G,c;\T) & := \EE(\frac{\alpha(G,c;\T)}{\sqrt{Tr(CC^{\dag})}})  = \Sen(\frac{CC^{\dag}}{Tr(CC^{\dag})}) \\
           & =  -\frac{Tr(CC^{\dag}\log(CC^{\dag}))}{Tr(CC^{\dag})}+ \log(Tr(CC^{\dag})).
\end{eqnarray}
% \onote{Missing Tr in the numerator of the last equation?}
Let $\MEE(G,c)$ be the maximum value of $\EE(G,c;\T)$ over all $\T'$s. The following lemma shows the logarithm of the quantum min-cut naturally provides an upper bound for the $\MEE(G,c)$, which is well known in literature. See for instance \cite{orus2014}.

\begin{lem}
\label{lem:ee<max}
Let $(G,c)$ be as above, then $\MEE(G,c)\leq \log \QMC(G,c)$.
\begin{proof}
For any tensor assignment $\T$, let the rank of $\beta(G,c;\T)$ be $r$, namely, the rank of $C$ is $r$ with the notations introduced above. Then the rank of $\rho_S = \frac{CC^{\dag}}{Tr(CC^{\dag})}$ is at most $r$. Let $\lambda_1, \cdots, \lambda_k$ be the non-zero eigenvalues of $\rho_S$, where $k \leq r$ and $Tr(\rho_S) = \sum\limits_{i=1}^{k} \lambda_i = 1$. Then $\EE(G,c;\T) = -\sum\limits_{i=1}^{k} \lambda_i \log \lambda_i \leq \log k \leq \log r \leq \log \QMF(G,c)$, with the first \lq\lq $=$'' holds if and only if $\lambda_i = \frac{1}{k}$ for all $i$. Thus $\MEE(G,c) \leq \log \QMF(G,c) \leq \log \QMC(G,c)$, where the last inequality is given by Corollary \ref{lem:max<min}.
\end{proof}
\end{lem}

We show that for the graphs considered in Theorem \ref{thm:max=min}, the inequality in Lemma \ref{lem:ee<max} is saturated.

\begin{thm}
\label{thm:ee=max}
For $(G,c)$, such that the capacity on each edge is a power of $d$, for some fixed integer $d$, then $\MEE(G,c) = \log \QMC(G,c)$.
\begin{proof}
%We show that the classical solution in the proof of Theorem \ref{thm:max=min} gives the maximum entanglement entropy, and we also use the notations in that proof.
In the proof of Theorem \ref{thm:max=min}, a particular linear map $\beta(G,c;\T)$ is produced which realizes the maximum rank $\QMF(G,c) = \QMC(G,c)$. Moreover, the matrix of $\beta(G,c;\T)$ under some appropriate orthogonal basis of $V_S$ and $V_T$ is given by that in Equation \ref{mat:beta}.  Therefore, the matrix of $\beta \beta^{\dag}$ is diagonal, of rank $\QMC(G,c)$, with either $1$ or $0$ on the diagonal. It follows that $\EE(G,c;\T) = \log \QMC(G,c)$. Combining with Lemma \ref{lem:ee<max}, we have $\MEE(G,c) = \log \QMC(G,c)$.
\end{proof}
\end{thm}

\section{Quantum max flow and quantum satisfiability}
\label{sec:qsat}

The quantum max flow is related to the quantum satisfiability problem, $\qsat$, introduced by Bravyi~\cite{bravyi2011efficient}. $\qsat$ is the following problem: you are given a hypergraph $G=(V,E)$, a labeling of the vertices  which we interpret as qudits dimensions $d:V \rightarrow \N$, $v\rightarrow d_v$, ranks labeling for hyperedges $r:E \rightarrow \N , \ e \rightarrow r_e$, and a (classical description of)  projectors $\Pi: e \rightarrow \Pi_e$ where each projector $\Pi_e$ acts only on the qudits in $e$ (i.e. $\Pi_e : \bigotimes_{v \in e} \C^{d_v} \rightarrow \bigotimes_{v \in e} \C^{d_v}$), and has rank $r_e$. We define the Hamiltonian
\begin{equation}
H = \sum_{e\in E} \Pi_e ^{(e)} \otimes I ^{( [n] \setminus e)}
\label{eq:Hamiltonian}
\end{equation}
The task is to decide whether $\ker(H)$ is non-trivial, i.e. whether there exists a state $\ket{\psi} \neq 0$ s.t. $H \ket{\psi}= 0$. (Since all the terms in the Hamiltonian $H$ are  positive semidefinite, $0$ is the smallest possible eigenvalue.)%We say that the instance is satisfiable iff $\ker(H)$ is non-trivial,  Otherwise, we say it is unsatisfiable
\footnote{For complexity theoretic reasons, the smallest eigenvalue of an unsatisfiable $\qsat$ instance is at least $1/p(n)$, where $n$ is the number of qudits in the system for some polynomial $p$. This requirement is irrelevant for this work.} What is the minimal value $\dim \ker(H)$ can take when $G,d,r$ are held fixed? By the rank-nullity theorem, this is equivalent (up to an additive factor) to the question: what is the maximal value $\rk(H)$ can take? Perhaps not surprisingly, the answer is the same as Proposition~\ref{prop:open dense}: a generic choice for the projectors gives a distinct value almost surely (with respect to the random choice of projectors), and this value minimizes $\dim\ker (H)$ among all possible choices for the projectors~\cite[Section 3.3.1]{laumann2010statistical}.

We define the function $\GQSAT(G,d,r)$ where $G,d,r$ are as before, but the projectors are not given explicitly, and this function is $\dim \ker (H)$ (see Eq.~\eqref{eq:Hamiltonian}) for a generic choice for the projectors. Note that the appropriate instance is generically unsatisfiable if and only if this function is $0$. A lower-bound on $\GQSAT$ is given in Ref.~\cite{sattath2015when}, and several other papers have studied when $\qsat$ is generically satisfiable (or, in our terminology, when $\GQSAT$ is strictly positive)~\cite{bravyi2009bounds,laumann2010product,ambainis2010quantum,movassagh2010unfrustrated,coudron2012unfrustration}.

We do not know whether there is a polynomial time reduction between $\QMF(\cdot)$ and $\GQSAT(\cdot)$. Nevertheless, there are specific cases where these problems are equivalent:

\begin{claim} $\GQSAT(G_2,d,r)= d_3(d_1 d_2 - r_1) -  \QMF(G_3,c)$ where $(G_2,d,r)$ is depicted in Fig.~\ref{fig:three_qudits} and $(G_3,c)$ is depicted in Fig.~\ref{fig:three_qudits_tensor_network}.
\label{cl:gen_qsat_eq_qmf}
\end{claim}

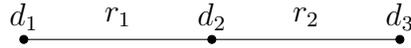
\begin{figure}[htp]
\begin{tikzpicture}[scale = 0.5]
\draw (0,0) -- (10,0) ;
\filldraw (0,0) circle(3pt);
\filldraw (5,0) circle(3pt);
\filldraw (10,0) circle(3pt);
\draw (0,0.6) node{$d_1$};
\draw (5,0.6) node{$d_2$};
\draw (10,0.6) node{$d_3$};
\draw (2.5,0.6) node{$r_1$};
\draw (7.5,0.6) node{$r_2$};
\end{tikzpicture}

\caption{A $\GQSAT$ instance with 3 qudits (vertices), and 2 projectors (edges). The dimensions of the qudits and the ranks of the projectors are parametrized by $d_1,d_2,d_3$ and $r_1,r_2$ respectively.}
\label{fig:three_qudits}
\end{figure}

\begin{figure}[htp]
\begin{tikzpicture}[scale = 0.5]
\draw (1,0) -- (8,0) ;
\draw (1,3) -- (8,3) ;
\draw (5,0) -- (4,3) ;

\draw (-0.6,3) node{$d_1 d_2 -r_1 $};
\draw (0,0) node{$d_3$};
\draw (4,3.5) node{$\T_1$};
\draw (5,-0.5) node{$\T_2$};
\draw (4,1.5) node{$d_2$};
\draw (9,3) node{$d_1$};
\draw (9,0) node{$r_2$};

\draw (-3,1.5) node{{\small{Input}}};
\draw (12,1.5) node{{\small{Output}}};
\end{tikzpicture}
\caption{$(G_3,c)$.}
\label{fig:three_qudits_tensor_network}
\end{figure}
\begin{proof}
Observe the cut depicted in Fig.~\ref{fig:three_qudits_tensor_network_with_cut}. By choosing the top tensor $\T_1$ appropriately, the image of the tensor network until this cut is the set of states which satisfy the first constraint acting on qudits 1 and 2. We can choose the bottom tensor $\T_2$ so that its kernel is the allowed subspace of the second projector acting on qudits 2 and 3. In this way, the dimension of the kernel of this entire tensor network is precisely the dimension of the satisfying subspace.
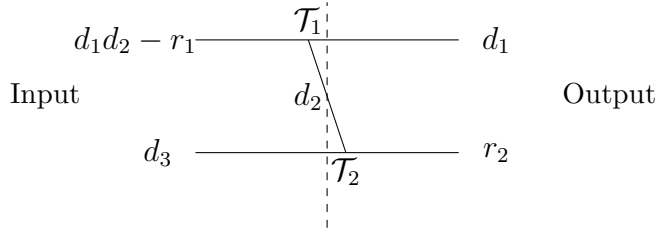
\begin{figure}[htp]
\begin{tikzpicture}[scale = 0.5]
\draw (1,0) -- (8,0) ;
\draw (1,3) -- (8,3) ;
\draw (5,0) -- (4,3) ;

\draw (-0.6,3) node{$d_1 d_2 -r_1 $};
\draw (0,0) node{$d_3$};
\draw (4,3.5) node{$\T_1$};
\draw (5,-0.5) node{$\T_2$};
\draw (4,1.5) node{$d_2$};
\draw (9,3) node{$d_1$};
\draw (9,0) node{$r_2$};

\draw (-3,1.5) node{{\small{Input}}};
\draw (12,1.5) node{{\small{Output}}};

%The cut
\draw[dashed] (4.5,-2) -- (4.5,4);
\end{tikzpicture}
\caption{A specific cut of the tensor network which is used in the proof of Claim~\ref{cl:gen_qsat_eq_qmf}.}
\label{fig:three_qudits_tensor_network_with_cut}
\end{figure}
By the rank-nullity theorem, the nullity is equal to $d_3(d_1d_2-r_1)$ minus its rank.
\end{proof}
Let us focus on the special case $d_1=d_3=3, d_2=2, r_1=1,r_2=5$%, as depicted in Fig.~\ref{fig:3_2_3_1_5}
. It can easily be checked that the quantum min-cut of the network is $15$. We will give two different proofs that the quantum max flow is $14$. From now on, we will use the convention for tensor networks that inputs are on the left and outputs are on the right.
\begin{align*}
\QMF\left(\begin{tikzpicture}[baseline={([yshift=-.5ex]current bounding box.center)},scale = 0.2]
\draw (1,0) -- (8,0) ;
\draw (1,3) -- (8,3) ;
\draw (4.5,0) -- (4.5,3) ;
\draw (0,3) node{$5$};
\draw (0,0) node{$3$};
\draw (3.5,1.5) node{$2$};
\draw (9,3) node{$3$};
\draw (9,0) node{$5$};
\end{tikzpicture}
\right) &=
3(3 \cdot 2 -1) - \GQSAT\left(
\begin{tikzpicture}[baseline={([yshift=-.5ex]current bounding box.center)},scale = 0.3]
\draw (0,0) -- (10,0) ;
\filldraw (0,0) circle(6pt);
\filldraw (5,0) circle(6pt);
\filldraw (10,0) circle(6pt);
\draw (0,1) node{$3$};
\draw (5,1) node{$2$};
\draw (10,1) node{$3$};
\draw (2.5,1) node{$1$};
\draw (7.5,1) node{$5$};
\end{tikzpicture}
\right) \\
&= 12 + 3(3 \cdot 2 - 5) - \GQSAT\left(\begin{tikzpicture}[baseline={([yshift=-.5ex]current bounding box.center)},scale = 0.3]
\draw (0,0) -- (10,0) ;
\filldraw (0,0) circle(6pt);
\filldraw (5,0) circle(6pt);
\filldraw (10,0) circle(6pt);
\draw (0,1) node{$3$};
\draw (5,1) node{$2$};
\draw (10,1) node{$3$};
\draw (2.5,1) node{$5$};
\draw (7.5,1) node{$1$};
\end{tikzpicture}\right) \\
&= 12 + \QMF\left( \begin{tikzpicture}[baseline={([yshift=-.5ex]current bounding box.center)},scale = 0.2]
\draw (1,0) -- (8,0) ;
\draw (1,3) -- (8,3) ;
\draw (4.5,0) -- (4.5,3) ;
\draw (0,3) node{$1$};
\draw (0,0) node{$3$};
\draw (3.5,1.5) node{$2$};
\draw (9,3) node{$3$};
\draw (9,0) node{$1$};
\end{tikzpicture}
\right) \\
&= 12 + \QMF\left( \begin{tikzpicture}[baseline={([yshift=-.5ex]current bounding box.center)},scale = 0.2]
\draw (1,0) -- (4.5,0) ;
\draw (4.5,3) -- (8,3) ;
\draw (4.5,0) -- (4.5,3) ;
%\draw (0,3) node{$1$};
\draw (0,0) node{$3$};
\draw (3.5,1.5) node{$2$};
\draw (9,3) node{$3$};
%\draw (9,0) node{$1$};
\end{tikzpicture}
\right) \\
&= 14.
%\label{eq:}
\end{align*}
Here we used Claim~\ref{cl:gen_qsat_eq_qmf} in the first and third steps.
%First, by Claim~\ref{cl:gen_qsat_eq_qmf}, the kernel of the tensor network is equal to $\GQSAT(G_2,d=(3,2,3),r=(1,5) )$. Furthermore, $\GQSAT(G_2,d=(3,2,3),r=(1,5) ) = \GQSAT(G_2,d=(3,2,3),r=(5,1) ) $  (here we replaced the roles of $r_1$ and $r_2$).
%
%By applying Claim~\ref{cl:gen_qsat_eq_qmf} again, $\GQSAT(G_2,d=(3,2,3),r=(5,1) ) $ is equal to the kernel of the tensor network in Fig.~\ref{fig:3_2_3_1_5_dual}. For the tensor network in Fig.~\ref{fig:3_2_3_1_5_dual} the quantum min cut is equal to the quantum max flow, which is 2, and therefore the dimension of the kernel is 1, and substituting the result once again we get that the quantum max flow is $14$.
%
In the second proof we show that
\[ \GQSAT\left(
\begin{tikzpicture}[baseline={([yshift=-.5ex]current bounding box.center)},scale = 0.3]
\draw (0,0) -- (10,0) ;
\filldraw (0,0) circle(6pt);
\filldraw (5,0) circle(6pt);
\filldraw (10,0) circle(6pt);
\draw (0,1) node{$3$};
\draw (5,1) node{$2$};
\draw (10,1) node{$3$};
\draw (2.5,1) node{$1$};
\draw (7.5,1) node{$5$};
\end{tikzpicture}
\right)=1\]
directly. Fix a generic choice for the projectors. Since there is a 5-dimensional constraint on qudits 2 and 3 (which have dimension 2 and 3 respectively), there exists a unique state (up to a phase) $\ket{\psi}$ such that $\Pi_2 \ket{\psi} = 0$. Since $\Pi_1$ has rank-1, we can always write it as $\Pi_1 = \ketbra{\varphi}{\varphi}$. We can use the Schmidt decomposition to express $\ket{\varphi} = \sum_{i=1}^2 \lambda_i \ket{\alpha_i} ^{(1)} \otimes \ket{\beta_i}^{(2)}$. (Note that the sum goes from $1$ to $\min(d_2,d_3)=2$, and not $3$.) The states $\{ \ket{\alpha_1}, \ket{\alpha_2} \}$ can be completed to an orthonormal basis by some state $\ket{\alpha_3}$. The state $\ket{\Omega} = \ket{\alpha_3}^{(1)} \otimes \ket{\psi}^{(2,3)}$ satisfies $\Pi_1^{(1,2)} \otimes I ^{(3)} \ket{\Omega} = \Pi_2^{(2,3)} \otimes  I ^{(1)} \ket{\Omega}=0$. It can  also be shown that this is the only state in the kernel (for generic choices), which proves that $\GQSAT(G_2,d=(3,2,3),r=(1,5) )=1$.
%\begin{figure}[htp]
%%\includegraphics[width=13cm]{3_2_3_1_5}
%\begin{tikzpicture}[scale = 0.5]
%\draw (1,0) -- (8,0) ;
%\draw (1,3) -- (8,3) ;
%\draw (5,0) -- (4,3) ;
%
%\draw (-0.6,3) node{$5 $};
%\draw (0,0) node{$3$};
%\draw (4,3.5) node{$\T_1$};
%\draw (5,-0.5) node{$\T_2$};
%\draw (4,1.5) node{$2$};
%\draw (9,3) node{$3$};
%\draw (9,0) node{$5$};
%
%
%\draw (-3,1.5) node{{\small{Input}}};
%\draw (12,1.5) node{{\small{Output}}};
%
%\end{tikzpicture}
%\caption{The tensor network of Claim~\ref{cl:gen_qsat_eq_qmf} where $d_1=d_3=2,\ r_1=1,\ r_2=5$.}
%\label{fig:3_2_3_1_5}
%\end{figure}
%
%\begin{figure}[htp]
%%\includegraphics[width=13cm]{3_2_3_1_5_dual}
%\begin{tikzpicture}[scale = 0.5]
%\draw (1,0) -- (8,0) ;
%\draw (1,3) -- (8,3) ;
%\draw (5,0) -- (4,3) ;
%
%\draw (-0.6,3) node{$1 $};
%\draw (0,0) node{$3$};
%\draw (4,3.5) node{$\T_1$};
%\draw (5,-0.5) node{$\T_2$};
%\draw (4,1.5) node{$2$};
%\draw (9,3) node{$3$};
%\draw (9,0) node{$1$};
%
%
%\draw (-3,1.5) node{{\small{Input}}};
%\draw (12,1.5) node{{\small{Output}}};
%\end{tikzpicture}
%\caption{The tensor network of Claim~\ref{cl:gen_qsat_eq_qmf} where $d_1=d_3=2,\ r_1=5,\ r_2=1$.}
%\label{fig:3_2_3_1_5_dual}
%\end{figure}

%There are several interesting properties for the example in Fig.~\ref{fig:three_qudits}: the kernel is non-empty (min flow = max-cut? ) only if
\begin{claim}
%For $r_1 \leq d_1 d_2, r_2 \leq d_2 d_3, r_3 \leq d_3 d_4$,
\begin{multline*}
\GQSAT\left(\begin{tikzpicture}[baseline={([yshift=-.5ex]current bounding box.center)},scale = 0.3]
\draw (0,0) -- (9,0) ;
\filldraw (0,0) circle(6pt);
\filldraw (3,0) circle(6pt);
\filldraw (6,0) circle(6pt);
\filldraw (9,0) circle(6pt);
\draw (0,1) node{$d_1$};
\draw (3,1) node{$d_2$};
\draw (6,1) node{$d_3$};
\draw (9,1) node{$d_4$};
\draw (1.5,1) node{$r_1$};
\draw (4.5,1) node{$r_2$};
\draw (7.5,1) node{$r_3$};
\end{tikzpicture}
\right)
= (d_1 d_2 -r_1) (d_3 d_4 -r_3) \\ -  \QMF\left(
\begin{tikzpicture}[baseline={([yshift=-.5ex]current bounding box.center)},scale = 0.3]
\draw (1,0) -- (8,0) ;
\draw (1,3) -- (8,3) ;
\draw (3,0) -- (6,1.5) -- (8,1.5) ;
\draw (3,3) -- (6,1.5) ;
\draw (-1.5,0) node{$d_3 d_4 -r_3$};
\draw (-1.5,3) node{$d_1d_2 - r_1$};
\draw (9,0) node{$d_4$};
\draw (9,1.5) node{$r_2$};
\draw (9,3) node{$d_1$};
\draw (4.5,0.75) node{$d_3$};
\draw (4.5,2.25) node{$d_2$};
%\draw (-1,1.5) node{{\small{Input}}};
%\draw (10,1.5) node{{\small{Output}}};
\end{tikzpicture}
\right)
\end{multline*}
% where $(G_4,d,r)$ is depicted in Fig.~\ref{fig:three_qudits_tensor_network} and $(G_5,c)$ is depicted in Fig.~\ref{fig:three_qudits_tensor_network}.
\label{cl:gen_qsat_eq_qmf2}
\end{claim}
The proof follows the same steps as of Claim~\ref{cl:gen_qsat_eq_qmf}, and is thus omitted.
This claim allows us to give an alternative proof that the quantum max flow of Example~\ref{eg:rk7 8} is indeed 7. The network in Fig.~\ref{fig:rk8} is a special case of the tensor network in the above claim when we choose $d_1=d_2=d_3=d_4=2,\ r_1=r_3=1,\ r_2=2$, and switch the roles of the inputs and outputs (which has no effect on the quantum max flow, see Remark~\ref{re:replace input output}).

By claim~\ref{cl:gen_qsat_eq_qmf2}, we only need to show that
\[ \GQSAT\left(\begin{tikzpicture}[baseline={([yshift=-.5ex]current bounding box.center)},scale = 0.5]
\draw (0,0) -- (9,0) ;
\filldraw (0,0) circle(6pt);
\filldraw (3,0) circle(6pt);
\filldraw (6,0) circle(6pt);
\filldraw (9,0) circle(6pt);
\draw (0,1) node{$2$};
\draw (3,1) node{$2$};
\draw (6,1) node{$2$};
\draw (9,1) node{$2$};
\draw (1.5,1) node{$1$};
\draw (4.5,1) node{$2$};
\draw (7.5,1) node{$1$};
\end{tikzpicture}
\right)
 = 2.\]
For a generic choice of the projectors, there always exists an invertible operator $P=P_1 \otimes P_2 \otimes P_3 \otimes P_4$ so that
\[ (P_1 \otimes P_2) \Pi_1 (P_1^{-1} \otimes P_2^{-1}) = (P_3 \otimes P_4) \Pi_3 (P_3^{-1} \otimes P_4^{-1}) = \ketbra{\psi^-}{\psi^-} \]
and
\[ (P_2 \otimes P_3) \Pi_2 (P_2^{-1} \otimes P_3^{-1})=\ketbra{\psi^-}{\psi^-} + \ketbra{\psi^+}{\psi^+},\]
where $\ket{\psi^\pm} = \frac{1}{\sqrt{2}}\ket{01} \pm \ket{10}$, $\Pi_1$ is the projector acting on the two leftmost qubits, $\Pi_2$ on the two center qubits, and $\Pi_3$ on the two rightmost qubits.
This is achieved by using $P_2$ and  $P_3$ to "fix" $\Pi_2$(note that the rank of $\Pi_2$ is $2$), $P_1$ to "fix" $\Pi_1$, and $P_4$ to "fix" $\Pi_3$ using the simple transformation described in~\cite{laumann2010product,bravyi2011efficient}. It is easy to check that $\ker(P H P^{-1}) =\spn \{\ket{0000}, \ket{1111} \}$, which implies that $\dim \ker (H) = 2$.

\paragraph{\bf An open question.} Ref.~\cite{sattath2015when} gives a lowerbound on $\GQSAT$ and they conjecture that their lowerbound is tight in some appropriate limit. An analogous conjecture in our setting might be:
\begin{equation}
 \forall G , c,\  \lim_{n \rightarrow \infty} \QMF(G, nc) = \QMC(G,nc) ,
\label{eq:conj}
\end{equation}
where $nc$ means multiplying all the capacities by $n$. One can interpret a result as the above equation as saying that a quantum phenomenon ($\QMF(\cdot) \neq \QMC(\cdot)$) disappears in a large system.
We do not know whether Eq.~\ref{eq:conj} implies the tightness conjecture of Ref.~\cite{sattath2015when} or vice versa.

\section*{Acknowledgement}

We thank Matthew Hastings for offering many valuable and insightful suggestions to this paper.

\bibliographystyle{plain}
\bibliography{maxflowmincutbib}

\appendix
\section{Proofs}
We prove Lemma \ref{lem:map dense} and the result in Example \ref{eg:rk3 4} in the following subsections, respectively.

\subsection{Proof of Lemma \ref{lem:map dense}}
\label{sec:proof lemma dense}
\begin{lem}\cite{dur2000three}
Let $U,V,W$ be vector spaces isomorphic to $\C^2$. Then the set of linear maps $\Phi: U \longrightarrow V \otimes W$, which can be written in the form $|1\rangle_U \mapsto |1\rangle_V\otimes|1\rangle_W$, $|2\rangle_U \mapsto |2\rangle_V\otimes|2\rangle_W$ under appropriate bases of $U,V$ and $W$, is an open dense subset of $Hom(U, V \otimes W)$.
\begin{proof}
For a matrix $A \in M_2(\C)$, denote the $i$-th row, the $j$-th column and the $(i,j)$-entry of $A$ by $A(i,\cdot)$, $A(\cdot,j)$ and $A(i,j)$, respectively.

Fix an arbitrary basis $\{|1\rangle, |2\rangle\}$ of $U,V$ and $W$. Then any linear map $\Phi: U \longrightarrow V \otimes W$ has the form:
\begin{multline} \label{equ:Phi}
 |1\rangle \mapsto A(1,1)|11\rangle + A(1,2)|12\rangle+A(2,1)|21\rangle+A(2,2)|22\rangle,\\|2\rangle \mapsto B(1,1)|11\rangle + B(1,2)|12\rangle+B(2,1)|21\rangle+B(2,2)|22\rangle.
\end{multline}

Let $A_{\Phi} = (A(i,j))_{1 \leq i,j \leq 2}, B_{\Phi} = (B(i,j))_{1 \leq i,j \leq 2} \in M_{2}(\C)$, then $Hom(U, V \otimes W)$ is isomorphic to $M_2(\C) \times M_2(\C)$ with each map $\Phi$ corresponding to $(A_{\Phi},B_{\Phi})$.

Let $\A \subset M_2(\C) \times M_2(\C)$ contain all the pairs $(A,B)$ which satisfies $(1)-(3)$:
%\onote{Same comment as before. Use enumerate?}
\begin{enumerate}
\item $\det(A) \neq 0$;

%$2).$ for any $1 \leq i\leq 2$, $A(i,\cdot)$ and $B(i,\cdot)$ are linearly independent in $\C^2$;

%$2).$ for any $1 \leq j \leq 2$, $A(\cdot,j)$ and $B(\cdot,j)$ are linearly independent in $\C^2$;

\item $A^{-1}B$ has two distinct eigenvalues, or equivalently, $(tr(A^{-1}B))^2 - 4 \det(A^{-1}B) \neq 0$;

\item $(A^{-1}B)(1,2) \neq 0$, $(A^{-1}B)(2,1) \neq 0. $
\end{enumerate}
%It follows from $2)$ or $4)$ that $A^{-1}B$ is not a diagonal matrix, and in particular, $A, B$ are linearly independent in $M_2(\C)$. Since $A^{-1}B$ is invertible, neither of the two eigenvalues of $A^{-1}B$ is $0$.

A direct consequence of $(3)$ is that $A, B$ are linearly independent in $M_2(\C)$. It is clear that the complement of $\A$ is a proper subvariety of $M_2(\C) \times M_2(\C)$. By \cite{hartshorne1977algebraic}, $\A$ is an open dense subset of $M_2(\C) \times M_2(\C)$. We prove that each pair $(A,B) \in \A$ gives a linear map $\Phi$ satisfying the property in the statement of the lemma.

Let $\lambda_1, \lambda_2$ be the two distinct eigenvalues of $A^{-1}B$, and define $D_i = \lambda_i I - A^{-1}B, i = 1,2$. Then $D_1$, $D_2$ are both non-zero matrices of rank $1$.

$$D_i =
\begin{pmatrix}
\lambda_i - (A^{-1}B)_{11} &  -(A^{-1}B)_{12} \\
-(A^{-1}B)_{21}            & \lambda_i - (A^{-1}B)_{22} \\
\end{pmatrix}
 %\quad
% D_2 =
%\begin{pmatrix}
%\lambda_2 - (A^{-1}B)_{11} &  -(A^{-1}B)_{12} \\
%-(A^{-1}B)_{21}            & \lambda_2 - (A^{-1}B)_{22} \\
%\end{pmatrix}
$$
Thus there are non-zero column vectors $u_1, u_2$ and non-zero row vectors $v_1, v_2$ such that $D_1 = u_1.v_1$ and $D_2 = u_2.v_2$. Moreover, noting that $D_i(1,2) = -(A^{-1}B)(1,2) \neq 0, D_i(2,1) = - (A^{-1}B)(2,1) \neq 0 $, we have that $u_i$ is proportional to $D_i(\cdot, 1)$ and $v_i$ is proportional to $D_i(1,\cdot)$.

Since $\lambda_1 \neq \lambda_2$, and $(A^{-1}B)(1,2) \neq 0, (A^{-1}B)(2,1) \neq 0$, we have $D_1(\cdot,1),D_2(\cdot,1)$ are linearly independent and $D_1(1,\cdot),D_2(1,\cdot)$ are linearly independent. Therefore, $\{u_1,u_2\}$ and $\{v_1,v_2\}$ are each a basis of $\C^2$. Since $A$ is invertible, $\{A.u_1, A.u_2\}$ is also a basis of $\C^2$.

Now we define a linear map $\Phi: U \longrightarrow V \otimes W$ by the pair $(A,B)$ according to Equation \ref{equ:Phi}. Then the coordinate of $\Phi(\lambda_i |1\rangle - |2\rangle)$, written in the matrix form under the basis $\{|j\rangle |k\rangle: 1 \leq j,k\leq 2\}$ of $V \otimes W$, is $\lambda_i A - B = A(\lambda_i - A^{-1}B) = A.u_i.v_i, \, i = 1,2$. Let $|i\rangle_V \in V$ be the vector with coordinate $A.u_i$, $|i\rangle_W \in W$ be the vector with coordinate $v_i$, and let $|i\rangle_U = \lambda_i |1\rangle - |2\rangle \in U$, then $\{|1\rangle_U,|2\rangle_U\}$, $\{|1\rangle_V,|2\rangle_V\}$, $\{|1\rangle_W,|2\rangle_W\}$ are basis of $U,V,W$, respectively, and $\Phi(|i\rangle_U) = |i\rangle_V\otimes|i\rangle_W$.

\end{proof}
\end{lem}

\subsection{Proof of $\QMF(G_1,c_1,L_1) \leq 3$ in Example \ref{eg:rk3 4}}
\label{subsec:example rk 3}

We prove that the tensor network $(G_1,c_1,L_1)$ in Example \ref{eg:rk3 4} has maximal rank at most $3$. See also Figure \ref{fig:rk3 copy}. It is shown below that for a generic choice of tensor $\T$, the resulting linear map $\beta(G_1,c_1,L_1;\T)$ has rank at most $3$. It can be proved in the same way as Proposition \ref{prop:open dense} that the set of tensors which realize $\QMF(G_1,c_1,L_1)$ is an open dense subset. Therefore, $\QMF(G_1,c_1,L_1) \leq 3 < \QMC(G_1,c_1)$.

\begin{figure}
\begin{tikzpicture}[scale = 0.5]
\begin{scope}
\draw (3,3) -- (6,0);
\draw [color = white, line width = 2mm](3,0) -- (6,3);
\draw (3,0) -- (6,3);
\draw (0,3) -- (9,3);
\draw (0,0) -- (9,0) ;

\filldraw (3,0) circle(3pt);
\filldraw (3,3) circle(3pt);
\filldraw (6,0) circle(3pt);
\filldraw (6,3) circle(3pt);

\draw (3,-0.8) node{$\T$};
\draw (3,3.8) node{$\T$};
\draw (6,-0.8) node{$\T$};
\draw (6,3.8) node{$\T$};

\draw (2.5,0) node{{\tiny{$1$}}};
\draw (3.5,0) node{{\tiny{$2$}}};
\draw (3.3,0.4) node{{\tiny{$3$}}};

\draw (5.5,0) node{{\tiny{$3$}}};
\draw (6.5,0) node{{\tiny{$1$}}};
\draw (5.7,0.4) node{{\tiny{$2$}}};

\draw (2.5,3) node{{\tiny{$1$}}};
\draw (3.5,3) node{{\tiny{$2$}}};
\draw (3.3,2.6) node{{\tiny{$3$}}};

\draw (5.5,3) node{{\tiny{$3$}}};
\draw (6.5,3) node{{\tiny{$1$}}};
\draw (5.7,2.6) node{{\tiny{$2$}}};

\end{scope}
\end{tikzpicture}
\caption{$(G_1,c_1,L_1)$}\label{fig:rk3 copy}
\end{figure}
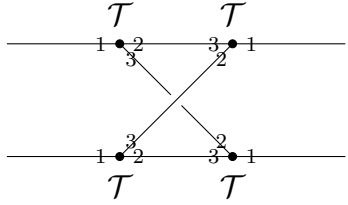

Let $\T = \{T_{ijk}: i,j,k = 0,1\}$ be a generic tensor, by Corollary \ref{cor:map dense}, there exist invertible tensors $A = \{A_{ij}: i,j = 0,1\}, B = \{B_{ij}: i,j = 0,1\}, C = \{C_{ij}: i,j = 0,1\}$, such that the equality in Figure \ref{tensorTS1} holds, where $\Sen = \{\Sen_{ijk}: i,j,k = 0, 1\} \in (\C^2)^{\otimes 3}$ is the tensor such that $\Sen_{ijk} = 1$ if $i=j=k$, and $\Sen_{ijk} = 0$ otherwise. For the readers' convenience, we display the equality in Figure \ref{tensorTS}.

\begin{figure}
\begin{tikzpicture}[scale = 0.5]
\begin{scope}
\draw (0,0) -- (2,0) -- (4,2);
\draw (2,0) -- (4,-2);
\draw (2,0.8) node{$\T$};
\draw (1.7,0) node{{\tiny{$1$}}};
\draw (2.3,0.3) node{{\tiny{$2$}}};
\draw (2.3,-0.3) node{{\tiny{$3$}}};
\filldraw (2,0) circle(3pt);
\end{scope}

\draw (6,0) node{$=$};

\begin{scope}[xshift = 8cm]
\draw (0,0) -- (3,0) -- (6,3);
\draw (3,0) -- (6,-3);
\draw (3,0.8) node{$\Sen$};
\draw (1.5,0.8) node{$A$};
\draw (4.5,2.3) node{$B$};
\draw (4.5,-2.3) node{$C$};

\filldraw (3,0) circle(3pt);
\filldraw (1.5,0) circle(3pt);
\filldraw (4.5,1.5) circle(3pt);
\filldraw (4.5,-1.5) circle(3pt);

\draw (2.7,0) node{{\tiny{$1$}}};
\draw (3.3,0.3) node{{\tiny{$2$}}};
\draw (3.3,-0.3) node{{\tiny{$3$}}};

\draw (1.2,0) node{{\tiny{$1$}}};
\draw (1.8,0) node{{\tiny{$2$}}};

\draw (4.2,1.2) node{{\tiny{$1$}}};
\draw (4.8,1.8) node{{\tiny{$2$}}};

\draw (4.2,-1.2) node{{\tiny{$1$}}};
\draw (4.8,-1.8) node{{\tiny{$2$}}};

\end{scope}
\end{tikzpicture}
\caption{Tensors $\T$ and $\Sen$} \label{tensorTS}
\end{figure}
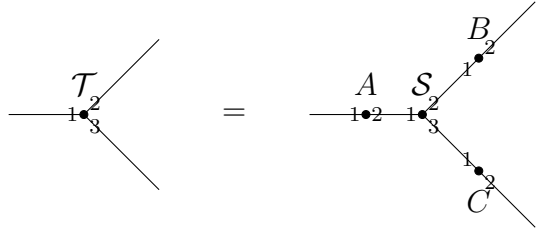

With this equality, the network in Figure \ref{fig:rk3 copy} with a generic tensor $\T$ is equal to the left network shown in Figure \ref{rank3_2}, which produces the map $\beta(G_1,c_1,L_1;\T)$. Let $D$ be the tensor $D = \{D_{ij}: D_{ij} = \sum\limits_{k} B_{ik}C_{jk}, i,j = 0,1\}$. Since the tensor $A$ is an invertible matrix, the rank of $\beta(G_1,c_1,L_1;\T)$ is not changed when the $A\,'$s on the two ends of the left network in Figure \ref{rank3_2} are removed, which results in the right network shown in Figure \ref{rank3_2}. Denote by $\Phi$ the linear map produced by the resulting network. It is straightforward that, when viewed as a linear map from $\C^2 \otimes \C^2$ to $\C^2 \otimes \C^2$, $\Phi |i,j\rangle = \sum\limits_{k,l} F(i,j;k,l) |k,l\rangle,  $ where $F(i,j;k,l) = D_{ik}D_{kj}D_{jl}D_{li}$. One can check directly that $F(i,j;k,l) = F(j,i;k,l)$, for any $k,l = 0,1$. Thus, $\Phi |0,1\rangle = \Phi |1,0\rangle$, and hence $rank(\Phi) = rank(\beta(G_1,c_1,L_1;\T))$ is at most $3$, which implies $\QMF(G_1,c_1,L_1) \leq 3$.

\begin{figure}
\begin{tikzpicture}[scale = 0.5]
\begin{scope}
\draw (3,4.5) -- (7.5,0);
\draw [color = white, line width = 2mm](3,0) -- (7.5,4.5);
\draw (3,0) -- (7.5,4.5);
\draw (0,4.5) -- (10.5,4.5);
\draw (0,0) -- (10.5,0) ;

\filldraw (3,0) circle(3pt);
\filldraw (3,4.5) circle(3pt);
\filldraw (7.5,0) circle(3pt);
\filldraw (7.5,4.5) circle(3pt);
\filldraw (1.5,0) circle(3pt);
\filldraw (4.5,0) circle(3pt);
\filldraw (6,0) circle(3pt);
\filldraw (9,0) circle(3pt);
\filldraw (1.5,4.5) circle(3pt);
\filldraw (4.5,4.5) circle(3pt);
\filldraw (6,4.5) circle(3pt);
\filldraw (9,4.5) circle(3pt);
\filldraw (4.5,1.5) circle(3pt);
\filldraw (4.5,3) circle(3pt);
\filldraw (6,1.5) circle(3pt);
\filldraw (6,3) circle(3pt);

\draw (3,-0.8) node{$\Sen$};
\draw (3,5.3) node{$\Sen$};
\draw (7.5,-0.8) node{$\Sen$};
\draw (7.5,5.3) node{$\Sen$};
\draw (1.5,-0.8) node{$A$};
\draw (1.5,5.3) node{$A$};
\draw (9,-0.8) node{$A$};
\draw (9,5.3) node{$A$};
\draw (4.5,-0.8) node{$B$};
\draw (4.5,5.3) node{$B$};
\draw (6.7,1.5) node{$B$};
\draw (6.7,3) node{$B$};
\draw (3.8,1.5) node{$C$};
\draw (3.8,3) node{$C$};
\draw (6,-0.8) node{$C$};
\draw (6,5.3) node{$C$};

\draw (2.7,0) node{{\tiny{$1$}}};
\draw (3.3,0) node{{\tiny{$2$}}};
\draw (3.4,0.4) node{{\tiny{$3$}}};

\draw (7.2,0) node{{\tiny{$3$}}};
\draw (7.8,0) node{{\tiny{$1$}}};
\draw (7.1,0.4) node{{\tiny{$2$}}};

\draw (2.7,4.5) node{{\tiny{$1$}}};
\draw (3.3,4.5) node{{\tiny{$2$}}};
\draw (3.4,4.1) node{{\tiny{$3$}}};

\draw (7.2,4.5) node{{\tiny{$3$}}};
\draw (7.8,4.5) node{{\tiny{$1$}}};
\draw (7.1,4.1) node{{\tiny{$2$}}};

\draw (1.2,0) node{{\tiny{$1$}}};
\draw (1.8,0) node{{\tiny{$2$}}};

\draw (4.2,0) node{{\tiny{$1$}}};
\draw (4.8,0) node{{\tiny{$2$}}};

\draw (5.7,0) node{{\tiny{$2$}}};
\draw (6.3,0) node{{\tiny{$1$}}};

\draw (8.7,0) node{{\tiny{$2$}}};
\draw (9.3,0) node{{\tiny{$1$}}};

\draw (1.2,4.5) node{{\tiny{$1$}}};
\draw (1.8,4.5) node{{\tiny{$2$}}};

\draw (4.2,4.5) node{{\tiny{$1$}}};
\draw (4.8,4.5) node{{\tiny{$2$}}};

\draw (5.7,4.5) node{{\tiny{$2$}}};
\draw (6.3,4.5) node{{\tiny{$1$}}};

\draw (8.7,4.5) node{{\tiny{$2$}}};
\draw (9.3,4.5) node{{\tiny{$1$}}};

\draw (4.2,1.2) node{{\tiny{$1$}}};
\draw (4.8,1.8) node{{\tiny{$2$}}};

\draw (4.2,3.3) node{{\tiny{$1$}}};
\draw (4.8,2.7) node{{\tiny{$2$}}};

\draw (5.7,1.8) node{{\tiny{$2$}}};
\draw (6.3,1.2) node{{\tiny{$1$}}};

\draw (5.7,2.7) node{{\tiny{$2$}}};
\draw (6.3,3.3) node{{\tiny{$1$}}};

\end{scope}

\begin{scope}[xshift = 12cm]
\draw (3,4.5) -- (7.5,0);
\draw [color = white, line width = 2mm](3,0) -- (7.5,4.5);
\draw (3,0) -- (7.5,4.5);
\draw (0,4.5) -- (10.5,4.5);
\draw (0,0) -- (10.5,0) ;

\filldraw (3,0) circle(3pt);
\filldraw (3,4.5) circle(3pt);
\filldraw (7.5,0) circle(3pt);
\filldraw (7.5,4.5) circle(3pt);
\filldraw (6,0) circle(3pt);
\filldraw (6,4.5) circle(3pt);
\filldraw (6,1.5) circle(3pt);
\filldraw (6,3) circle(3pt);

\draw (3,-0.8) node{$\Sen$};
\draw (3,5.3) node{$\Sen$};
\draw (7.5,-0.8) node{$\Sen$};
\draw (7.5,5.3) node{$\Sen$};

\draw (6.7,1.5) node{$D$};
\draw (6.7,3) node{$D$};
\draw (6,-0.8) node{$D$};
\draw (6,5.3) node{$D$};

\draw (2.7,0) node{{\tiny{$1$}}};
\draw (3.3,0) node{{\tiny{$2$}}};
\draw (3.4,0.4) node{{\tiny{$3$}}};

\draw (7.2,0) node{{\tiny{$3$}}};
\draw (7.8,0) node{{\tiny{$1$}}};
\draw (7.1,0.4) node{{\tiny{$2$}}};

\draw (2.7,4.5) node{{\tiny{$1$}}};
\draw (3.3,4.5) node{{\tiny{$2$}}};
\draw (3.4,4.1) node{{\tiny{$3$}}};

\draw (7.2,4.5) node{{\tiny{$3$}}};
\draw (7.8,4.5) node{{\tiny{$1$}}};
\draw (7.1,4.1) node{{\tiny{$2$}}};

\draw (5.7,0) node{{\tiny{$1$}}};
\draw (6.3,0) node{{\tiny{$2$}}};

\draw (5.7,4.5) node{{\tiny{$1$}}};
\draw (6.3,4.5) node{{\tiny{$2$}}};

\draw (5.7,1.8) node{{\tiny{$2$}}};
\draw (6.3,1.2) node{{\tiny{$1$}}};

\draw (5.7,2.7) node{{\tiny{$2$}}};
\draw (6.3,3.3) node{{\tiny{$1$}}};

\end{scope}
\end{tikzpicture}
\caption{$(G_1,c_1,L_1)$}\label{rank3_2}
\end{figure}
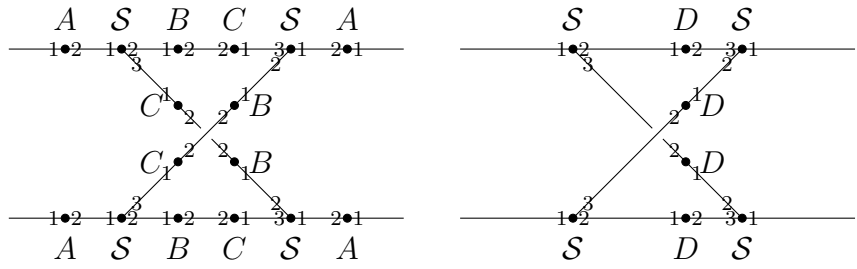
\end{document}